\documentclass[12pt]{amsart}
\usepackage{dsfont}
\usepackage{mathrsfs} 
\usepackage{amssymb, changes}
\usepackage{amsmath}
\usepackage{amsfonts}
\usepackage{bbm}
\usepackage{amsthm}
\usepackage{amsbsy}
\usepackage{amsgen}
\usepackage{amscd}
\usepackage{amsopn}
\usepackage{amstext}
\usepackage{amsxtra}
\usepackage[english]{babel}
\usepackage[most]{tcolorbox}
\usepackage{xcolor}
\usepackage{times}
\usepackage[foot]{amsaddr}
\usepackage{hyperref}\hypersetup{
    colorlinks=true,
    linkcolor=purple,
    filecolor=blue,      
    urlcolor=cyan,
    }
   
\usepackage[margin=45pt]{geometry}

\tcolorboxenvironment{theorem}{
  enhanced,
  breakable,
  colback=gray!3,
  colframe=black!50!black,
  boxrule=0.8pt,
  arc=2mm,
  left=6pt,
  right=6pt,
  top=6pt,
  bottom=6pt,
  before skip=10pt,
  after skip=10pt
}

\tcolorboxenvironment{lemma}{
  enhanced,
  breakable,
  colback=gray!3,
  colframe=black!40!black,
  boxrule=0.8pt,
  arc=2mm,
  left=6pt,
  right=6pt,
  top=6pt,
  bottom=6pt,
  before skip=10pt,
  after skip=10pt
}

\tcolorboxenvironment{proposition}{
  enhanced,
  breakable,
  colback=gray!3,
  colframe=black!50!black,
  boxrule=0.8pt,
  arc=2mm,
  left=6pt,
  right=6pt,
  top=6pt,
  bottom=6pt,
  before skip=10pt,
  after skip=10pt
}

\newtheorem{theorem}{Theorem}[section]

\newtheorem{lemma}[theorem]{Lemma}

\newtheorem{proposition}[theorem]{Proposition}

\newtheorem{definition}{Definition}

\newcommand{\E}{\mathbb{E}}
\renewcommand{\P}{\mathbb{P}}

\def\Tcal{\mathcal T}

\newcommand{\X}{\mathbf{X }}

\newcommand{\R}{\mathbb{R}}
\newcommand{\Z}{\mathbb{Z}}
\newcommand{\N}{\mathbb{N}}
\def\eps{\varepsilon}
\newcommand{\rmd}{\,\mathrm{d}}

\def\0{\mathbf{0}}

\title{``True" self-avoiding walks on general trees}
\author{Tuan-Minh Nguyen}
\address[T.-M. N.]{School of Mathematical Sciences, Monash University, 9 Rainforest Walk, Clayton 3800, Victoria, Australia}
\email{tuanminh.nguyen@monash.edu}

\begin{document}
\begin{abstract}
We study the asymptotic behavior of ``true" self-avoiding random walks on general infinite locally finite trees. In this model, the walk starts at the root and, at each step, from its current vertex chooses a neighboring edge to traverse with probability proportional to the current weight of that edge, where the weight of each edge after being traversed $n$ times is given by $w(n)=\exp(-\beta n)$. We show that the process exhibits a sharp phase transition between recurrence and transience. The critical value is determined by the branching-ruin number of the tree, which coincides with the Hausdorff dimension of the boundary of the tree under a suitable metric. We prove that the walk is almost surely transient when the branching-ruin number is greater than $1/2$, and recurrent when it is less than $1/2$. This resolves an open question posed by Kosygina.
\end{abstract}

\maketitle

\tableofcontents

\section{Introduction}

Let $\mathcal{T} = (V, E)$ be an infinite, locally finite tree with root $\rho \in V$.  
Let $\Z_+$ be the set of non-negative integers and $\N=\Z_{+}\setminus\{0\}.$ Fix $w: \mathbb{Z}_+ \to (0,\infty)$, which is called the weight function.  
Let $\X=(X_n)_{n \ge 0}$ be a discrete-time nearest-neighbor random walk taking values on $V$. The process starts at $X_0 = \rho$. For each undirected edge $\{x,y\} \in E$ and each time $n \ge 0$, let
\[
\mathsf L_n(x,y) = \sum_{k=0}^{n-1} \mathbf{1}_{\big\{\{X_k,X_{k+1}\}=\{x,y\}\big\}}.
\]
be the number of times the walk has crossed the edge $\{x,y\}$ (in either direction) up to time $n$.  For each $n$, let $\mathcal{F}_n$ be the $\sigma$-algebra generated by $X_0, X_1, \dots, X_n$. For two
adjacent vertices $x$ and $y$, we write $x\sim y$. Given $\mathcal{F}_n$ and that $X_n = x$, the conditional distribution of the next step is given by
\begin{align}\label{eq.trans}
   \mathbb{P}(X_{n+1} = y \mid \mathcal{F}_n) = \mathbf{1}_{\{y\sim x\}}\frac{w\bigl(\mathsf L_n(x,y)\bigr)}{\sum_{z \sim x} w\bigl(\mathsf L_n(x,z)\bigr)}. 
\end{align}
When $w(n) = \exp(-\beta n)$ for some $\beta > 0$, the process is called the \textbf{``true" self-avoiding walk} (TSAW). We say that the process $\mathbf X$ is
\begin{itemize}
    \item \textbf{recurrent} if every vertex of $\mathcal T$ is visited infinitely often;
    \item \textbf{transient} if every vertex of $\mathcal T$ is visited only finitely often.
\end{itemize}

The objective of the paper is to determine recurrence and transience for the
TSAW on trees of polynomial growth. The relevant geometric quantity is the
\textbf{branching-ruin number}, introduced in \cite{CKS2020}. For each edge $e=\{v^{-1},v\}\in E$, let $|e|=|v|$ be the distance from $v$ to $\rho$, i.e., the number of edges in the shortest path connecting $v$ and $\rho$. A \textbf{cutset} in $\mathcal{T}$ is a minimal set $\pi$ of edges that separates the root from infinity. That is, for any infinite self-avoiding path $(v_i)_{i\ge 0}$ with $v_0=\rho$, there exists a unique $i$ such that $\{v_i, v_{i+1}\} \in \pi$.
Let $\Pi$ be the set of all cutsets in $\mathcal{T}$.
 The branching-ruin number of $\Tcal$ is given by
\begin{equation}\label{brr2}
\text{br}_r(\mathcal{T}) = \sup\Big\{\gamma > 0: \inf_{\pi\in \Pi}\sum_{e\in\pi}|e|^{-\gamma} > 0\Big\}.
\end{equation}
This is the polynomial analogue of the branching number introduced by Lyons
\cite{Lyons1989} to measure trees with exponential growth.
Let $\partial \mathcal{T}$ stand for the \textbf{boundary} of $\mathcal{T}$, which is the set of all infinite self-avoiding paths starting from the root. We define the metric between any two infinite paths $\xi, \eta \in \partial \mathcal{T}$ whose last common edge is $e$ by
\begin{equation}\label{def.d}
    d(\xi, \eta) = 1/|e|.
\end{equation}
and $d(\xi,\eta)=1$ if they have no common edge. 
Note that ${\rm br}_r(\mathcal{T})$ is equal to the \textbf{Hausdorff dimension} of $\partial \mathcal{T}$ with respect to the metric $d$ (see Section 3.3 in \cite{CKS2020}).

A typical example of a tree with polynomial growth is the $\mathbb{Z}^d$-like tree $T_d$, defined as follows. A vertex has $d$ children if $|v|=2^k$ for some $k\in\mathbb{Z}_+$, and exactly one child otherwise. Then $\text{br}_r(T_d)=\log_2(d).$ 

In this paper, we establish a criterion of recurrence and transience for the TSAW on any infinite local finite tree $\mathcal{T}$ with respect to the branching-ruin number $\text{br}_r(\mathcal{T})$. The main result of this paper resolves a conjecture proposed by Kosygina \cite{ABHR2025} that TSAWs on trees with polynomial growth exhibit a phase transition between recurrence and transience.  

\begin{theorem}\label{thm:main}
The ``true" self-avoiding walk on $\Tcal$ is recurrent if ${\rm br}_r(\Tcal)<1/2$ and transient if ${\rm br}_r(\Tcal)>1/2$.
\end{theorem}

The TSAW with site repulsion on the integer lattice $\mathbb{Z}^d$ was first introduced and studied by Amit, Parisi, and Peliti \cite{APP1983} as a dynamic model of polymer growth, in which newly added
monomers preferentially avoid previously visited sites. Non-rigorous scaling
and renormalization-group arguments predict dimension-dependent asymptotic
scaling behaviour for this process \cite{APP1983, OP1983, PP1987}. Beyond polymer physics, the TSAW has found applications in diverse areas, including network exploration \cite{KPY2016, CMC2021}, quantum algorithms
\cite{CRT2014}, and biological chemotaxis \cite{A2021, BBV2022, RG2026}. It has also emerged as a fundamental tool for understanding the large-scale behaviour of non-reversible sampling algorithms such as Event-Chain Monte Carlo \cite{M2025}.

Tóth \cite{Toth1995} subsequently introduced a bond-repulsion variant of the TSAW, which enabled the first rigorous analysis of the model on the one-dimensional lattice $\mathbb{Z}$. In higher dimensions, diffusive scaling
limits for the continuous-time TSAW with site repulsion on $\mathbb{Z}^d$, $d \ge 3$, were
established by Horváth, Tóth, and Vető \cite{HTV2010}. The continuous-time scaling limit, known as the \emph{true self-repelling motion}, was introduced by Tóth and Werner \cite{TW1998}. More recently, Kosygina and Peterson \cite{KP2026} proved that suitable rescalings of the TSAW
converge weakly to this continuous process. To the best of our knowledge, rigorous studies of TSAWs on graphs beyond the integer lattice are largely absent from the literature. Thus, our results on general trees provide a step toward understanding the asymptotic behavior of TSAWs on broader classes
of graphs.

The remainder of the paper is organized as follows.
 In Section~\ref{sec:strong.constr}, we present the Rubin’s construction of the process $\X$ together with the extension processes defined along geodesic paths of $\mathcal{T}$. Each extension process has the same law as the restriction of $\X$ to the corresponding path, conditioned to be visited infinitely often, and hence has the same law as the one-dimensional TSAW. In Section~\ref{sec:1D-TSAW}, we analyze the one-dimensional TSAW on $\{0,1,\dots,n\}$ and derive an exact asymptotic formula for the \textit{ruin probability} that the walk hits $n$ before returning to $0$. The key ingredient is a comparison between the Markov chain $Y$, which records the number of backward steps before hitting $n$, and a symmetric random walk $S$ whose increment law is the stationary law of $Y$. We study the local behavior of the chains $Y$ and $S$  conditioned to stay above a barrier $K$. Using a Duhamel expansion, we compare the first return probabilities of $Y$ and $S$ to the finite set $\{0,1,\dots,K\}$, for $K$ sufficiently large, and thereby obtain the corresponding asymptotics for $Y$. The exact asymptotic formula for the ruin probabilities then follows from a Markov renewal formula. In Section~\ref{sec:proof-main}, we use these ruin events to define a \textit{percolation} on $\mathcal{T}$ by declaring an edge $e=\{v^{-1},v\}$ open if the extension process along the geodesic from the root to $v$ hits $v$ before returning to the root, and closed otherwise. This percolation captures the transience/recurrence behavior of the walk: almost sure finiteness of the open cluster containing the root implies recurrence, whereas positivity of the probability that this cluster is infinite implies transience. We prove that this percolation is quasi-independent in the sense of Lyons \cite{Lyons1989}. The proof of quasi-independence is based on estimating the joint probability
that two edges $e_1$ and $e_2$ are open, conditioned on their last common
ancestral edge $e$ being open. The open events of $e_1$ and $e_2$ are
correlated through the local times accumulated on their common ancestral
segment. We condition on the number of crossings of the last common edge $e$
and estimate the contributions coming from the two remaining geodesic
segments, from $e$ to $e_1$ and from $e$ to $e_2$. These estimates are
obtained by analyzing the one-dimensional Markov chains that record the
numbers of forward crossings. We show that the dependence created by the
common ancestral segment can be controlled uniformly, thereby establishing
quasi-independence. This in turn allows us to analyze the above percolation
model and deduce Theorem~\ref{thm:main}.

It is worth noting that quasi-independent percolation was introduced by Lyons
\cite{Lyons1989} and has since been used as a standard framework in the analysis of several models of random walks on trees, including random walks in random environments \cite{LP1992, CNT2026}, once-reinforced random walks \cite{CKS2020}, random walks among random conductances \cite{CHK2019}, and once-excited random walks \cite{LN2026}. In these models, the relevant ruin probabilities and the quasi-independence of the associated percolation can often be established directly by coupling with the classical gambler's ruin problem for the birth-and-death process on $\Z_+$. For TSAWs, however, this coupling technique no longer applies due to the strong dependence on the past trajectory of the process. We believe that our approach, based on the analysis of the Markov chains recording the number of backward steps and forward steps, is not specific to TSAWs and may extend to a broader class of self-interacting processes on trees.

\section{Strong construction}\label{sec:strong.constr}

In this section we construct the TSAW using a family of independent exponential clocks and define the restriction and extension processes that will reduce the analysis on a tree to the one-dimensional model on a path. For a vertex $v\neq\rho$, we denote its parent by $v^{-1}$. For each edge $e=\{v^{-1},v\}\in E$, denote by $\mathcal{P}_e$ or $\mathcal{P}_v$ the unique shortest path of edges connecting $v$ to $\rho$. For two edges $e_1=\{v_1^{-1},v_1\}$ and $e_2=\{v_2^{-1},v_2\}$, we write $e_1 \le e_2$ or $e_1\le v_2$ if $e_1\in \mathcal{P}_{e_2}$. We also write $e_1 < e_2$ or $e_1< v_2$ if $e_1\le e_2$ and $e_1\neq e_2$. 

\subsection{Rubin's construction}\label{sec:rubin}
 Let $\vec{E}$ be the set of all oriented edges induced from $E$. Let $$\boldsymbol{\xi} = \big( \xi(x, y, j): \, (x,y)\in \vec{E}, j \geq 0\big)$$ be a collection of independent exponential random variables with rate 1. We use $\boldsymbol{\xi}$ to give a strong construction of the process $\mathbf{X}$ satisfying the transition law given by \eqref{eq.trans} as follows.
\begin{itemize}
    \item Set $X_0 = \rho$.
    \item Assume that $(X_k)_{1\le k\le n}$ has been defined. Let $C_n(x,y):=\sum_{j=0}^{n-1}\mathbf{1}_{\{X_j=x,\ X_{j+1}=y\}}$ be the number of crossings from $x$ to $y$ up to time $n$. Let $$T_n(x,y)=\sum_{k=0}^{{C}_n(x,y)}\frac{\xi(x,y, k)}{w(2k+\mathbf{1}_{\{x^{-1}=y\}})}.$$
    On the event $\{X_n=x\}$,
the next position is given by
$$X_{n+1}= \arg\min_{y\sim x} T_{n}(x,y).$$
\end{itemize}
 
This construction is inspired by the Rubin's construction for the generalized Pólya urn (see Section 5 in \cite{D1990}).

\subsection{Restrictions and extensions}

For a connected subset $B\subset V$ and $n\ge 1$, let $$\delta_n(B):=\inf\Big\{ k\ge 0: \sum_{j=0}^k\mathbf{1}_{\{X_j\in B\}}=n\Big\} \quad\text{and}\quad s_{B}:=\sup\{n\ge 1 : \delta_n(B)<\infty\},$$
where we adopt the conventions $\inf \emptyset = \infty$ and $\sup\emptyset =-\infty$. 
Define $m_0 = 1$ and $m_{k+1} = \min\{ j > m_k : X_{\delta_j(B)} \neq X_{\delta_{m_k}(B)} \}$ for each $k\ge0$. Let $$K_B:=\sup\{k\ge 0: m_k\le s_{B}\}.$$
We call $(X_{\delta_{m_n}(B)})_{0\le n\le K_B}$ the \textbf{restriction} of the process $\X$ to $B$. This is a nearest-neighbor random walk which describes the movement of the process $\X$ within $B$, ignoring times when the process $\X$ is outside $B$. We call $K_B$ the \textbf{killing time} of the restriction to $B$. 

Fix a vertex $v \in V\setminus\{\rho\}$. Recall that $\mathcal{P}_v$ is the shortest path connecting $\rho$ and $v$. We now construct a process $\mathbf{X}^{(v)} = (X_n^{(v)})_{n \ge 0}$ on $\mathcal{P}_v$, which is coupled with $\X$ such that $\X^{(v)}$ has the same trajectory as the restriction of $\X$ on $\mathcal{P}_v$ conditioned on the event that $\X$ visits $\mathcal{P}_v$ infinitely many times.
This process is defined as follows.
\begin{itemize}
    \item Set $X_0^{(v)}=\rho$.
    \item  Let ${C}^{(v)}_n(x,y)$ be the number of crossings from $x$ to $y$ by the process $\X^{(v)}$ up to time $n$.  For $x\in\mathcal{P}_v\setminus\{v\}$, denote by $(x_i)_{1\le i\le \deg(x)-1}$ the children of $x$. There exists a unique $j\in\{1,2,\cdots, \deg(x)-1\}$ such that $x_{j} \in\mathcal{P}_v$. On the event $\{X_n^{(v)}=x\}$ with $x\neq v$, the next position is defined by
$$X_{n+1}^{(v)}= \arg\min_{y\in \{ x^{-1}, x_{j}\}}\left\{ \sum_{k=0}^{{C}^{(v)}_n(x,y)}\frac{\xi(x,y, k)}{w(2k+\mathbf{1}_{\{x^{-1}=y\}})}\right\}.$$
On the event $\{X_n^{(v)}=v\}$, set $X_{n+1}^{(v)}=v^{-1}$.
\end{itemize}
By the above construction, we immediately obtain the following restriction principle.
\begin{lemma}[Restriction principle]\label{lem:rest}
The process $\X^{(v)}$ constructed above is a TSAW on $\mathcal{P}_v$. Furthermore, a.s.
\begin{align*}\label{coind}
    {X}^{(v)}_n=X_{\delta_{m_n}(\mathcal{P}_v)}\quad \text{for all $0\le n\le K_{\mathcal{P}_v}$,}
\end{align*}
i.e., $\X^{(v)}$ coincides with the restriction of $\X$ to $\mathcal{P}_v$ up to the killing time $K_{\mathcal{P}_v}$.
\end{lemma}
\section{``True" self-avoiding walks on a path}\label{sec:1D-TSAW}

Fix $n\ge 1$ and consider the extension process $\X^{(v)}$ on the path
$\mathcal{P}_v$ with $|v|=n$. Notice that the process $\X^{(v)}$ has the same distribution as the one-dimensional TSAW on $\{0,1,\cdots, n\}$. We denote the latter process by $\widetilde{\X}=(\widetilde X_k)_{k\ge 0}$. This process starts from $\widetilde X_0=0$, jumps deterministically from $0$ to
$1$, and jumps deterministically from $n$ to $n-1$. For each
$1\le x\le n-1$, when the process is at $x$, the probability of jumping to
$x-1$ or $x+1$ is proportional to the current weights of the two adjacent
edges, where the weight of an edge traversed $m$ times is $w(m)=e^{-\beta m}$. 

For $m\in\{0,1,\ldots,n\}$, slightly abusing notation, we let
$$
\tau_m:=\inf\{k\ge 0: \widetilde X_k=m\} \quad\text{and} \quad\tau_0^+:=\inf\{k\ge 1: \widetilde X_k=0\}
$$
be respectively the first hitting time of $m$ and the first return time to $0$. In this section, we aim to study the asymptotic behavior of the ruin probability
$$
r_n:=\P(\tau_n<\tau_0^+).
$$

\subsection{Markovian structure of backward steps}\label{sec:backward-chain}

For each $x\in\{1,\ldots,n\}$, define
$$
B(x,n):=\sum_{k=0}^{\tau_n-1}\mathbf{1}_{\{\widetilde X_k=x,\;\widetilde X_{k+1}=x-1\}},
$$
the number of backward crossings from $x$ to $x-1$ before the first hit of $n$.
Clearly,
$
B(n,n)=0.
$ 
Notice also that the ruin probability is also given by
$$
r_n=\P(B(1,n)=0).
$$
For every integer $u$, define
$$
p(u):=\frac{e^{-\beta(2u+1)}}{1+e^{-\beta(2u+1)}},
\quad
q(u):=1-p(u)=\frac{1}{1+e^{-\beta(2u+1)}}.
$$
Notice that $0<p(u)<1$ for all $u\in\mathbb{Z}$. Let $(\eta_n)_{n\ge0}$ be a Markov chain on $\mathbb{Z}$ with transition probabilities
$$
P(u,v):=\mathbb P(\eta_{n+1}=v\mid \eta_n=u)=\begin{cases}
q(v+1)\prod_{r=u}^{v} p(r), & \text{if } u-1\le v,\\
0, & \text{otherwise}.
\end{cases}
$$
Define the Markov chain $Y=(Y_n)_{n\ge0}$ on $\mathbb Z_+$ with $Y_0=0$ whose transition probabilities are defined by
$$
Q(z,y):=\mathbb P(Y_{n+1}=y\mid Y_n=z)=P^{\,z+1}(0,y-z-1),
\quad y\ge0.
$$
The following lemma shows that the backward sequence $(B(n,n-k))_{0\le k\le  n-1}$ has the same law as the Markov chain $(Y_k)_{0\le k\le n-1}$.
\begin{lemma} \label{prop:auxiliary-transition}
For $n\ge1$,
$$
(B(n,n),B(n-1,n),\dots,B(1,n))
\stackrel d=(Y_0,Y_1,\dots,Y_{n-1}).
$$
Consequently, 
$$
r_n=\mathbb P_0(Y_{n-1}=0).
$$
\end{lemma}
\begin{proof}
It is sufficient to show that for every $x\in\{2,\dots,n\}$ and every $y,z\ge0$, we have
\begin{align*}
\mathbb P(B(x-1,n)=y\mid B(x,n)=z)=P^{\,z+1}(0,y-z-1),
\end{align*}
where $P^{\,z+1}$ denotes the $(z+1)$-step transition probability of the Markov chain $\eta$ defined above. We adapt an idea by Kesten-Kozlov-Spitzer for nearest-neighbor random walks \cite{Kesten} which was later used for bond‑repelling walks on $\Z$ by Toth in \cite{Toth1995}. Fix $x\in\{2,\dots,n\}$. By definition, $B(x,n)$ is the number of backward jumps $x\to x-1$ before the first hit of $n$. Since the walk must cross the edge $\{x-1,x\}$ one more time forward than backward in order to get from $0$ to $n$ before time $\tau_n$, the number of forward crossings $x-1\to x$ up to time $\tau_n$ is exactly $B(x,n)+1$.

Conditional on $\{B(x,n)=z\}$, there are therefore $z+1$ forward jumps from $x-1$ to $x$ by time $\tau_n$. Each such forward jump from $x-1$ to $x$ is called a ``failure'', and each jump from $x-1$ to $x-2$ is called a ``success''. For $1\le j\le z+1$, let $\ell_j(x,n)$ be the number of successes that occur after the $(j-1)$-st failure and before the $j$-th failure, where for $j=1$ this means before the first failure. Then clearly
\begin{align*}
B(x-1,n)=\sum_{j=1}^{z+1}\ell_j(x,n).
\end{align*}
We now compute the law of the sequence $(\ell_j(x,n))_{1\le j\le z+1}$. Set
\[
u_0:=0,\quad u_j:=\sum_{i=1}^j\ell_i(x,n)-j\quad\text{for } j\ge1.
\]
Thus $u_j=u_{j-1}+\ell_j(x,n)-1$. After the first $j-1$ failures and the first $\sum_{i=1}^{j-1}\ell_i(x,n)$ successes have occurred, the edge $\{x-2,x-1\}$ has been crossed exactly
$
2\sum_{i=1}^{j-1}\ell_i(x,n)+1
$
times, whereas the edge $\{x-1,x\}$ has been crossed exactly
$
2(j-1)
$
times. Therefore, whenever the walk is at $x-1$ during the $j$-th stage, the probability that the next move is a success is
\[
\frac{e^{-\beta\left(2\sum_{i=1}^{j-1}\ell_i(x,n)+1\right)}}{e^{-\beta\left(2\sum_{i=1}^{j-1}\ell_i(x,n)+1\right)}+e^{-2\beta(j-1)}}
=
\frac{e^{-\beta(2u_{j-1}+1)}}{1+e^{-\beta(2u_{j-1}+1)}}
=p(u_{j-1}),
\]
and the probability of a failure is $q(u_{j-1})$.
Hence, conditional on the past up to the beginning of the $j$-th stage,
\begin{align*}
&\mathbb P\bigl(\ell_j(x,n)=s_j \,\big|\, \ell_1(x,n)=s_1,\dots,\ell_{j-1}(x,n)=s_{j-1},\, B(x,n)=z\bigr)\\
&\quad=
q(u_{j-1}+s_j)\prod_{r=u_{j-1}}^{u_{j-1}+s_j-1}p(r)=\mathbb P\bigl(\eta_j=u_{j-1}+s_j-1 \mid \eta_{j-1}=u_{j-1}\bigr).
\end{align*}
Therefore,
\begin{align*}
&\P\bigl(\ell_1(x,n)=s_1,\dots,\ell_{z+1}(x,n)=s_{z+1}\mid B(x,n)=z\bigr)\\
&\quad=
\P\bigl(\eta_1-\eta_0=s_1-1,\;\dots,\;\eta_{z+1}-\eta_z=s_{z+1}-1\mid \eta_0=0\bigr).
\end{align*}
Summing over all sequences $(s_1,\dots,s_{z+1})$ with total $\sum_{j=1}^{z+1}s_j=y$, we obtain
\[
\P\big(B(x-1,n)=y\mid B(x,n)=z\big)
= \sum_{s_1+\dots+s_{z+1}=y}
\P\bigl(\eta_1-\eta_0=s_1-1,\dots,\eta_{z+1}-\eta_z=s_{z+1}-1\mid\eta_0=0\bigr).
\]
The right-hand side is exactly the probability that after $z+1$ steps the chain $\eta$, starting at $0$, is at position $y-(z+1)$. Hence
\[
\P(B(x-1,n)=y\mid B(x,n)=z)=P^{\,z+1}(0,\;y-z-1).
\]
This completes the proof.
\end{proof}

To study the process $(Y_n)_{n\ge 0}$, we use the following asymptotic result of the Markov chain $(\eta_n)_{n\ge0}$.
\begin{lemma}[Lemma 1 and Lemma 2 in \cite{Toth1995}]\label{lem:expconv}
 The unique stationary distribution of $(\eta_n)$ is 
\[
\varrho(x) := \frac{e^{-\beta(x+1)^2}}{\sum_{z \in \mathbb{Z}} e^{-\beta(z+1)^2}}, \quad x \in \mathbb{Z},
\]
whose variance is
\[
\varsigma^2_{\beta} := \frac{\sum_{z \in \mathbb{Z}} z^2 e^{-\beta z^2}}{\sum_{z \in \mathbb{Z}} e^{-\beta z^2}}. 
\] 
Furthermore, there exist positive constants $C$ and $c$ such that: 
\begin{itemize}
    \item[a.] For all \(n \ge 0\),
\[
\sum_{y \in \mathbb{Z}} |P^n(0,y) - \varrho(y)| \le C  e^{-c n}. 
\]
Consequently, $|\E[\eta_{n}\mid\eta_0=0]+1|\le C e^{-c n}$ and $|\E[\eta_{n}^2\mid\eta_0=0]-\varsigma^2_{\beta}|\le C e^{-cn}.$
\item[b.] For all $n\ge0$ and $x\ge0$, \[ P^n(0,x+1)\le C e^{-\beta x}P^n(0,x). \] 
\end{itemize} 
\end{lemma}

By the same proof as Lemma~\ref{lem:expconv}, the conclusion of the lemma remains valid if the initial state $0$ is replaced by any fixed $a\in\mathbb Z$, with constants that
may depend on $a$. 

It is natural to introduce the shifted stationary law
$$
\nu(j):=\varrho(j-1)=\frac{e^{-\beta j^2}}{\sum_{m\in\mathbb Z}e^{-\beta m^2}},
\quad j\in\mathbb Z.
$$
The law $\nu$ is symmetric and has finite variance $
\varsigma^2_{\beta}=\sum_{j\in\mathbb Z}j^2\nu(j).
$
For $z\ge0$, define
\begin{align*}
   \mu_z(j)&:=\mathbb P(Y_1-Y_0=j\mid Y_0=z)=P^{\,z+1}(0,j-1),
\quad \text{for }j\in\mathbb Z, \\
    m_1(z)&:=\mathbb E(Y_1-Y_0\mid Y_0=z)=\sum_{j\in\mathbb Z}j\,\mu_z(j),
\quad
m_2(z) :=\mathbb E((Y_1-Y_0)^2\mid Y_0=z)=\sum_{j\in\mathbb Z}j^2\,\mu_z(j).
\end{align*}
The next lemma collects the basic asymptotic properties of the jump law of $Y$, which will be used repeatedly in the sequel.
\begin{lemma}\label{lem:step_distr}
There exist constants $C<\infty$ and $c>0$ such that for each $z\ge0$,
\begin{align} \label{mu.tail}
   & \mu_z(j)\le C e^{-c j^2} 
\quad\text{for }j\ge 0,\\
   & \sum_{j\in\mathbb Z}|\mu_z(j)-\nu(j)|
\le Ce^{-cz}, \label{lem:limiting-jump-law}
\quad \\ 
\label{lem:uniform-tails} & \mathbb P(|Y_1-Y_0|\ge m\mid Y_0=z)\le Ce^{-cm}\quad\text{for }m\ge 0,\\ \label{mu.1momment}
  &|m_1(z)|\le Ce^{-cz},\\ \label{mu.2moment}
&\bigl|m_2(z)-\varsigma^2_{\beta}\bigr|\le Ce^{-cz}.
\end{align}
\end{lemma}

\begin{proof}
By Lemma \ref{lem:expconv}.b, we have
\begin{align*}
P^n(0,m)
&=
P^n(0,0)\prod_{r=0}^{m-1}\frac{P^n(0,r+1)}{P^n(0,r)}
\le
P^n(0,0)C_0^m \exp\!\Bigl(-\beta\sum_{r=0}^{m-1}r\Bigr) \le C e^{-c(m+1)^2}.
\end{align*}
For $n=z+1$ and $m=j-1$, we get for every $j\ge 1$,
$
\mu_z(j)=P^{\,z+1}(0,j-1)\le C e^{-c j^2}.
$
After enlarging $C$ if necessary, the same bound also holds for $j=0$.
Hence \eqref{mu.tail} is verified.

Recall that by definition, $
\nu(j)=\varrho(j-1)
$.
Therefore
\begin{align*}
\sum_{j\in\mathbb Z}|\mu_z(j)-\nu(j)|
&=
\sum_{j\in\mathbb Z}\bigl|P^{\,z+1}(0,j-1)-\varrho(j-1)\bigr|=
\sum_{m\in\mathbb Z}|P^{\,z+1}(0,m)-\varrho(m)|.
\end{align*}
By Lemma \ref{lem:expconv}.a, the last
quantity is bounded by $Ce^{-cz}$, which proves \eqref{lem:limiting-jump-law}.

Using the fact that $\nu$ has zero mean, we have
\begin{align*}
|m_1(z)|
&=
\Big|\sum_{j\in\mathbb Z}j\bigl(\mu_z(j)-\nu(j)\bigr)\Big|.
\end{align*}
Splitting the sum into $|j|\le z$ and $|j|>z$, we get
\begin{align*}
|m_1(z)|
&\le
z\sum_{j\in\mathbb Z}|\mu_z(j)-\nu(j)|
+\sum_{|j|>z}|j|\mu_z(j)
+\sum_{|j|>z}|j|\nu(j).
\end{align*}
By \eqref{lem:limiting-jump-law}, the first term is at most
$Cze^{-cz}$.
For the $\mu_z$-tail, note that $Y_1\ge 0$, so
$
Y_1-Y_0\ge -Y_0
$
and therefore
$
\mu_z(j)=0
$
for $j<-z$. Using this fact and \eqref{mu.tail}, we have
$$
\sum_{|j|>z}|j|\mu_z(j)=\sum_{j>z}|j|\mu_z(j)\le C_0\sum_{j>z}j\,e^{-c_0 j^2}\le Ce^{-cz}.$$
Since $\nu$ has Gaussian tails by definition,
$\sum_{|j|>z}|j|\nu(j)\le Ce^{-cz}$.
The two tail sums and the first term are all bounded by $Ce^{-cz}$ after absorbing the linear
factor into the exponential. This proves \eqref{mu.1momment}. The proof for $m_2(z)$ is analogous.
We write
$$
m_2(z)-\varsigma^2_{\beta}
=
\sum_{j\in\mathbb Z}j^2\bigl(\mu_z(j)-\nu(j)\bigr)
$$
and split the sum at $|j|\le z$. The bounded part is
controlled by \eqref{lem:limiting-jump-law}, while the tails are controlled
by Gaussian domination. This proves \eqref{mu.2moment}.

By definition,
$$
\mathbb P(|Y_1-Y_0|\ge m\mid Y_0=z)
=
\sum_{|j|\ge m}\mu_z(j).
$$
Using \eqref{mu.tail} and the fact that $
\mu_z(j)=0
$
for $j<-z$, we have
$
\sum_{|j|\ge m}\mu_z(j) \le Ce^{-c m}.
$
This proves \eqref{lem:uniform-tails}.
\end{proof}


\begin{lemma} \label{lem:recurrent} The Markov chain $(Y_n)_{n\ge0}$ is irreducible and recurrent.
\end{lemma}
\begin{proof}
We first show that $Y$ is irreducible on $\mathbb Z_+$. For every $z \ge 1$, we have 
$$
Q(0,z)=P(0,z-1)=q(z)\prod_{r=0}^{z-1}p(r)>0,
$$
Also, for every $z \ge 0$,
$$
Q(z,0)=P^{\,z+1}(0,-z-1)\ge \prod_{r=0}^{z}q(-r)>0,
$$
since the Markov chain $\eta$ can follow the path
$
0\to -1\to -2\to \cdots \to -(z+1)
$
with strictly positive probability. Hence, every state communicates with $0$, and
the Markov chain $Y$ is thus irreducible. 

By \eqref{mu.1momment} and \eqref{mu.2moment}, there exist constants $C,c>0$
such that
$$
m_1(z)\le C e^{-cz},
\quad
\bigl|m_2(z)-\varsigma^2_{\beta}\bigr|\le C e^{-cz},
\quad z \ge 0.
$$
Hence there exists $z_0 \in \mathbb N$ such that for all
$z \ge z_0$,
$
2z\,m_1(z)<m_2(z).
$
Moreover, by \eqref{mu.tail}, the conditional increment law of $Y$ has a Gaussian upper tail
uniformly in the current state. In particular, $\sup_{z\ge 0}\E[|Y_1-Y_0|^p \mid Y_0=z]<\infty$ for $p>2$. Applying the recurrence criterion of Markov chains with asymptotically zero drift (see e.g. Theorem 3.2.3 in \cite{MPW2017}, p. 94), we deduce that the chain $Y$ is null recurrent.
\end{proof}

We next construct a harmonic function for the process killed at $0$ associated with the Markov chain $Y$.

\begin{lemma}\label{lem:harmonic}
There exists a non-negative harmonic function $h:\mathbb Z_+\to [0,\infty)$ satisfying
$$
h(0)=0, \quad h(z)=z+O(1)
\text{ as }z\to\infty
\quad  \text{and} \quad
\mathbb E\left[h(Y_1)\mathbf 1_{\{Y_1>0\}}\mid Y_0=z\right]=h(z)
\text{ for }z\ge 1.
$$
\end{lemma}
\begin{proof} Let
\begin{align*}  \sigma_0 :=\inf\{n\ge 0:Y_n=0\}\quad\text{and}\quad
\widetilde \sigma_N:=\inf\{n\ge 0:Y_n\ge N\}\quad \text{for $N\ge 1$}. 
\end{align*}
Recall from Lemma \ref{lem:recurrent} that $Y$ is irreducible and recurrent on $\mathbb Z_+$. Hence for every $N\ge 1$ and
every $1\le x<N$, we have
$$
T_N:=\sigma_0\wedge\widetilde \sigma_N<\infty
\quad\text{a.s. under }\P_x.
$$

For $x\in \Z_+$, define
$\varphi_N(x):=\P_x(\widetilde \sigma_N<\sigma_0)$. Note that 
$\varphi_N(0):=0$ and $\varphi_N(x):=1$ for $x\ge N$.
Then, for every $1\le x\le N-1$,
\begin{equation}\label{eq:hN-harmonic}
\varphi_N(x)=\mathbb E_x\!\left[\varphi_N(Y_1)\mathbf 1_{\{Y_1>0\}}\right].
\end{equation}

We next proceed with the proof in 5 steps.

\smallskip
\noindent
{\bf Step 1: exponential test functions.} Set $\theta_x:=Y_1-x$.
By Lemma~\ref{lem:step_distr}, the jumps of $Y$ have uniformly exponential tails, so
there exists $\lambda_0>0$ such that
\begin{equation}\label{eq:uniform-exp-moment}
\sup_{x\ge 1}\mathbb E_x\!\left[e^{\lambda_0|\theta_x|}\right]<\infty.
\end{equation}
In particular,
$
\sup_{x\ge 1}\mathbb E_x\!\left[|\theta_x|^3e^{\lambda_0|\theta_x|}\right]<\infty.
$
Choose $\gamma\in(0,\lambda_0\wedge c)$ so small that
$$
\frac{\varsigma^2_{\beta}}{4}\gamma^2-C_1\gamma^3>0,
$$
where $C_1<\infty$ is such that the remainder term $R_x$ in
the Taylor expansion $
e^{-\gamma \theta_x}=1-\gamma \theta_x+\frac{\gamma^2}{2}\theta_x^2+R_x
$
satisfies
$
|R_x|
\le
\frac{\gamma^3}{6}|\theta_x|^3e^{\gamma|\theta_x|}$ and $\sup_{x\ge 1}\mathbb E_x|R_x|\le C_1\gamma^3.
$
Therefore,
\[
\mathbb E_x[e^{-\gamma \theta_x}]
=
1-\gamma m_1(x)+\frac{\gamma^2}{2}m_2(x)+O(\gamma^3),
\]
uniformly in $x$. Recall that by \eqref{mu.1momment} and \eqref{mu.2moment},
$
m_1(x)=O(e^{-cx})$ and $m_2(x)=\varsigma^2_{\beta}+O(e^{-cx}).$
Hence there exist constants $\eps>0$ and $K_0\in [1,\infty)$ such that
\begin{equation}\label{eq:exp-drift-full}
\mathbb E_x[e^{-\gamma (Y_1-x)}]\ge 1+2\eps
\quad\text{for all }x\ge K_0.
\end{equation}
Since, by Lemma~\ref{lem:step_distr}, $Y_1-Y_0$ has uniformly exponential
tails, we also have
$\P_x(Y_1=0)\le C e^{-cx}$,
for each $x\ge 1$.
After enlarging $K_0$ if necessary and using $\gamma<c$, it follows from
\eqref{eq:exp-drift-full} that
\begin{equation}\label{eq:exp-drift}
\mathbb E_x[e^{-\gamma (Y_1-x)}\mathbf 1_{\{Y_1>0\}}]\ge 1+\eps
\quad\text{for all }x\ge K_0.
\end{equation}

\smallskip
\noindent
{\bf Step 2: sub-harmonic and super-harmonic barriers.}
For each function $\phi: \Z_+\to\R$, let
$$
P_+\phi(x):=\E_x[\phi(Y_1)\mathbf{1}_{\{Y_1>0\}}].
$$
Let $I(x):=x$. Using \eqref{mu.1momment} and \eqref{mu.tail}, we notice that
$$
P_+I(x)-x
=
\E_x[(Y_1-x)\mathbf 1_{\{Y_1>0\}}]
=
m_1(x)+x\,\P_x(Y_1=0)
=
O(e^{-c'x})
$$
for some $c'>0$. Therefore, we can choose sufficiently large $A>0$ such that
\begin{equation}\label{eq:A-choice}
A\eps e^{-\gamma x}\ge 2|P_+I(x)-x|
\quad\text{for all }x\ge K_0.
\end{equation}
Define  $\psi_+(0)=0,\psi_-(0)=-A$ and
\begin{align*}
\psi_+(x):=x+A e^{-\gamma x},
\quad
\psi_-(x):=x-A e^{-\gamma x}, \quad \text{for $x\ge K_0$} 
\end{align*}
For $1\le x\le K_0-1$, define
$$
\psi_+(x):=\mathbb E_x\!\left[\psi_+(Y_{T_{K_0}})\right],
\quad
\psi_-(x):=\mathbb E_x\!\left[\psi_-(Y_{T_{K_0}})\right].
$$
These expectations are finite because \eqref{mu.tail} implies that the
overshoot at $T_{K_0}$ has finite mean. Hence, $\psi_+, \psi_- : \Z_+\to \R$ are well-defined. 
By \eqref{eq:exp-drift} and \eqref{eq:A-choice}, we notice that for $x\ge K_0$,
\begin{align*}
P_+\psi_+(x)-\psi_+(x)
&=
\left(P_+I(x)-x\right)
+A e^{-\gamma x}\bigl(\mathbb E_x[e^{-\gamma(Y_1-x)}\mathbf 1_{\{Y_1>0\}}]-1\bigr)
\ge 0,\\
P_+\psi_-(x)-\psi_-(x)
&=
\left(P_+I(x)-x\right)
-A e^{-\gamma x}\bigl(\mathbb E_x[e^{-\gamma(Y_1-x)}\mathbf 1_{\{Y_1>0\}}]-1\bigr)
\le 0.
\end{align*}
By the strong Markov property, for every $1\le x\le K_0-1$,
$$
P_+\psi_+(x)=\psi_+(x),
\quad
P_+\psi_-(x)=\psi_-(x).
$$
Hence, for all $x\ge 1$, 
\begin{equation}\label{eq:global-barriers}
P_+\psi_+(x)\ge \psi_+(x),
\quad
P_+\psi_-(x)\le \psi_-(x).
\end{equation}

\smallskip
\noindent
{\bf Step 3: a uniform overshoot bound.}
We claim that
\begin{equation}\label{eq:overshoot-bound}
\sup_{N\ge 1}\sup_{1\le x<N}
\mathbb E_x\!\left[(Y_{T_N}-N)^+;\ \widetilde \sigma_N<\sigma_0\right]
<\infty.
\end{equation}
Indeed, on $\{\widetilde \sigma_N<\sigma_0\}$, we have $Y_{T_N-1}<N$ and hence
$
Y_{T_N}-N\le (Y_{T_N}-Y_{T_N-1})^+.
$
Therefore
\begin{align*}
\mathbb E_x\!\left[(Y_{T_N}-N)^+;\ \widetilde \sigma_N<\sigma_0\right]
&\le
\sum_{n\ge 1}\mathbb E_x\!\left[(Y_n-Y_{n-1})^+;\ T_N=n\right]\\
&=
\sum_{n\ge 1}\mathbb E_x\!\left[
\mathbf 1_{\{T_N=n\}}
\mathbb E\!\left[(Y_n-Y_{n-1})^+\mid Y_{n-1}\right]\right]
\le
\sup_{z\ge 1}\mathbb E_z[(Y_1-z)^+].
\end{align*}
The last quantity is finite by the exponential tail bound in
Lemma~\ref{lem:step_distr}, proving \eqref{eq:overshoot-bound}.

\smallskip
\noindent
{\bf Step 4: finite-volume estimates.}
Fix $N>K_0$ and $1\le x<N$. First, since $\psi_-(Y_{n\wedge T_N})+A$ is a nonnegative supermartingale
by \eqref{eq:global-barriers}, by Fatou's lemma,
$$
\mathbb E_x[\psi_-(Y_{T_N})]\le \psi_-(x).
$$
On the event $\{\widetilde \sigma_N<\sigma_0\}$, we have $Y_{T_N}\ge N$, and therefore
$
\psi_-(Y_{T_N})\ge Y_{T_N}-A e^{-\gamma Y_{T_N}}\ge N-A.
$
On the event $\{\sigma_0<\widetilde \sigma_N\}$, we have $Y_{T_N}=0$ and hence
$
\psi_-(Y_{T_N})=-A.
$
Thus
$$
\mathbb E_x[\psi_-(Y_{T_N})]
\ge
(N-A)\varphi_N(x)-A(1-\varphi_N(x))
=
N \varphi_N(x)-A.
$$
Hence
\begin{equation}\label{eq:hN-upper}
N \varphi_N(x)\le \psi_-(x)+A.
\end{equation}
Similarly, since $\psi_+(Y_{n\wedge T_N})$ is a nonnegative submartingale, by
Fatou's lemma, we have
$$
\psi_+(x)\le \mathbb E_x[\psi_+(Y_{T_N})].
$$
On $\{\sigma_0<\widetilde \sigma_N\}$, we have $Y_{T_N}=0$ and $\psi_+(Y_{T_N})=0$. On
$\{\widetilde \sigma_N<\sigma_0\}$, we have
$
\psi_+(Y_{T_N})\le Y_{T_N}+A.
$
Therefore, using \eqref{eq:overshoot-bound},
$$
\mathbb E_x[\psi_+(Y_{T_N})]
\le
N \varphi_N(x)+A \varphi_N(x)
+\mathbb E_x\!\left[(Y_{T_N}-N)^+;\ \widetilde \sigma_N<\sigma_0\right]
\le
N \varphi_N(x)+C_2,
$$
for some constant $C_2<\infty$ independent of $x,N$. Hence
\begin{equation}\label{eq:hN-lower}
N \varphi_N(x)\ge \psi_+(x)-C_2.
\end{equation}

For $x\ge K_0$, we have
$
\psi_+(x)=x+A e^{-\gamma x}$ and
$\psi_-(x)=x-A e^{-\gamma x}$. Thus
 \eqref{eq:hN-upper}--\eqref{eq:hN-lower} imply that
\begin{equation}\label{eq:NhN-large}
x-C_3\le N \varphi_N(x)\le x+C_3,
\quad K_0\le x<N,
\end{equation}
for some constant $C_3<\infty$. For $1\le x<K_0$, \eqref{eq:hN-upper} implies
\begin{equation}\label{eq:NhN-small}
N \varphi_N(x)\le \max_{1\le y<K_0}\psi_-(y)+A=:C_4.
\end{equation}
For $x\in \Z_+$, define
$$
h_N(x):=N \varphi_N(x).
$$
Since also $h_N(x)\ge 0$, we obtain
$
|h_N(x)-x|\le \max\{C_4,K_0-1\}$ for all $1\le x<K_0$. Combining with \eqref{eq:NhN-large}, we conclude that there exists $C<\infty$ such that
\begin{equation}\label{eq:HN-uniform-profile}
|h_N(x)-x|\le C
\quad\text{for all }N\ge 2,\ 1\le x<N.
\end{equation}

\smallskip
\noindent
{\bf Step 5: compactness and passage to the limit.} By \eqref{eq:hN-upper} and \eqref{eq:HN-uniform-profile}, there exists
$C_5<\infty$ such that
\begin{equation}\label{eq:HN-linear}
0\le h_N(x)\le x+C_5
\quad\text{for all }x\ge 0,\ N\ge 1.
\end{equation}
Fix $x\ge 0$. Then the sequence $(h_N(x))_{N>x}$ is bounded by
\eqref{eq:HN-linear}. By a diagonal argument, there exists a subsequence $N_k\to\infty$
and a function $h:\mathbb Z_+\to[0,\infty)$ such that
$$
h_{N_k}(x)\to h(x)
\quad\text{for every }x\ge 0.
$$
Since $h_{N_k}(0)=0$, we have $h(0)=0$.

We now prove harmonicity. Fix $x\ge 1$. For all $k$ sufficiently large, $x<N_k$, and
\eqref{eq:hN-harmonic} gives
$$
h_{N_k}(x)=\mathbb E_x\!\left[h_{N_k}(Y_1)\mathbf 1_{\{Y_1>0\}}\right].
$$
By \eqref{eq:HN-linear},
$
0\le h_{N_k}(Y_1)\mathbf 1_{\{Y_1>0\}}\le Y_1+C_5.
$
Since $Y_1$ has finite first moment by Lemma~\ref{lem:step_distr}, using the dominated
convergence theorem, we get
$$
\lim_{k\to\infty}\mathbb E_x\!\left[h_{N_k}(Y_1)\mathbf 1_{\{Y_1>0\}}\right]
=
\mathbb E_x\!\left[h(Y_1)\mathbf 1_{\{Y_1>0\}}\right].
$$
Passing to the limit, we obtain
$$
h(x)=\mathbb E_x\!\left[h(Y_1)\mathbf 1_{\{Y_1>0\}}\right]
\quad\text{for every } x\ge 1.
$$

Finally, \eqref{eq:NhN-large} implies that for every fixed $x\ge K_0$ and all
$k$ sufficiently large,
$
|h_{N_k}(x)-x|\le C_3.
$
Letting $k\to\infty$, we obtain
$
|h(x)-x|\le C_3$ for each $x\ge K_0$. Thus
$
h(x)=x+O(1)$
as $x\to\infty$.
This completes the proof.
\end{proof}

\subsection{Killed symmetric random walk}\label{sec:benchmark}

In this subsection we introduce a killed symmetric walk associated with the
limiting law of the Markov chain $(Y_n)_{n\ge0}$ and collect the asymptotic
results that will be used throughout the sequel. Recall that the limiting law is
$$
\nu(j):=\frac{e^{-\beta j^2}}{\sum_{m\in\mathbb Z}e^{-\beta m^2}},
\quad j\in\mathbb Z.
$$
The distribution $\nu$ is symmetric and has finite variance
$
\varsigma^2_{\beta}:=\sum_{j\in\mathbb Z}j^2\nu(j).
$

Let $(\zeta_k)_{k\ge1}$ be i.i.d.\ random variables with common law $\nu$, and
define the symmetric random walk $S=(S_n)_{n\ge0}$ on $\mathbb Z$ by
\begin{align}\label{def:S}
    S_0=0,
\quad
S_n=\sum_{k=1}^n \zeta_k,
\quad n\ge1.
\end{align}
Fix $K\ge0$, and define
$$
W_K:=\{0,1,\dots,K\},
\quad
E_K:=\{K+1,K+2,\dots\}.
$$
Let
\begin{align} \label{def:Qbar}
 \bar Q(r,s):=\mathbb P(S_{n+1}=s\mid S_n=r)= \nu(s-r) \quad \text{for $r,s\in E_K$},
\end{align}
be the transition probabilities of $S$ on $E_K$.

Let $H$ be the ascending ladder-height renewal function given by
\begin{align}\label{def:H}
    H(u)
:=
\mathbf 1_{\{u>0\}}
+
\sum_{k=1}^{\infty}
\mathbb P\!\left(\chi_1^+ + \cdots + \chi_k^+ < u\right),
\quad u\in\mathbb R,
\end{align}
in which, $(\chi^+_k)_{k\ge 1}$ are i.i.d. copies of $\chi^+:=S_{T^+}$ with $T^+ := \min\{n\ge 1:S_n>0\}$. Note that
\begin{align}
    \label{H.asym}
    H(u)\sim \frac{u}{\E[\chi^+]} \quad \text{as $u\to\infty$}
\end{align}
Also, let 
$$
H_K(x):=H(x-K),
\quad x\in E_K.
$$
Define the excursion probabilities of $S$ in $E_K$ and its generating function by
\begin{equation}
    \label{R.first.return.kernel}
\begin{aligned}p_S^{(K)}(n;x,y)
&:=
\mathbb P_x^S(S_1,\dots,S_{n-1}\in E_K, S_n=y),
\quad x,y\in E_K, n\ge 1.
\\
\widehat{p}_S^{(K)}(s;x,y)&:=\sum_{n=0}^{\infty}p_S^{(K)}(n;x,y)s^{n}, \quad  x,y\in E_K, 0\le s<1,\end{aligned}
\end{equation}
with the convention $p_S^{(K)}(0;x,y):=\mathbf 1_{\{x=y\}}.$ The next result shows that the excursion probabilities of $S$ decays polynomially with exponent $3/2$.
\begin{lemma}
\label{prop:benchmark-first-return}
There exists a constant $C\in (0,\infty)$ such that for all $x,y\in E_K$ and $n\ge 1$,
$$p_S^{(K)}(n;x,y)\le C\,H_K(x+1)H_K(y)n^{-3/2},$$
Furthermore, for each $x,y\in E_K,$
\begin{align*}
&\widehat{p}_S^{(K)}(1;x,y)-\widehat{p}_S^{(K)}(s;x,y)
=
\frac{\sqrt2}{\varsigma_\beta}\,H_K(x+1)\,H_K(y)\,\sqrt{1-s}
+
o(H_K(x+1)\,H_K(y)\sqrt{1-s}),
\quad s\uparrow1,\\ 
&\widehat{p}_S^{(K)}(1;x,y)-\widehat{p}_S^{(K)}(s;x,y)
\le
C\,H_K(x+1)\,H_K(y)\,\sqrt{1-s},
\quad 0\le s<1.\end{align*}
\end{lemma}

\begin{proof}
Fix $x,y\in E_K$, and write
$x':=x-K,\ y':=y-K$, which are positive integers.
Using the translation invariance, we note that $p_S^{(K)}(m;x,y)$ is exactly the local
probability that the symmetric random walk $(S_n)_{n\ge0}$ starts from $x'$ and stays positive up to
time $m$ and is at $y'$ at time $m$, i.e.
\begin{align}\label{eq:half-line-identification}
p_S^{(K)}(m;x,y)
=
\mathbb P_{x'}^S(S_1>0,\dots,S_{m-1}>0,\ S_m=y').
\end{align}
Since $(S_n)_{n\ge0}$ is an aperiodic walk on $\Z$ whose common increment law $\nu$ has zero mean and finite variance $\varsigma^2_{\beta}$, the local probability of $S$ conditioned to stay positive has the exact asymptotic (see, e.g., Theorem 3 in \cite{DW2024}):
\begin{equation}\label{eq:benchmark-local-asymptotic}
\begin{aligned}
  p_S^{(K)}(m;x,y)
&\sim
\frac{H(x+1-K)\,H(y-K)}{\varsigma_\beta\sqrt{2\pi}}\,
m^{-3/2}
\exp\!\Big(-\frac{(x-y)^2}{2\varsigma^2_{\beta}m}\Big)\\
&\sim \frac{H_K(x+1)\,H_K(y)}{\varsigma_\beta\sqrt{2\pi}}\,
m^{-3/2}, \quad   m\to\infty,
\end{aligned}
\end{equation}
where we recall that $H$ is the renewal function given by \eqref{def:H} and $H_K(z):=H(z-K)$ for $z\in E_K$.

For $u, v\in \Z$ and $n \ge 1$, set
$$
q_n(u,v):=\mathbb P_u\bigl(S_1\neq0,\dots,S_{n-1}\neq0,\ S_n=v\bigr).
$$
By formula (1.5) in \cite{U2011}, there exists a constant
$C\in (0,\infty)$ such that
\begin{equation}\label{eq:uchiyama-origin-bound}
q_n(u,v)\le C\,u\,v\,n^{-3/2}
\quad \text{for all } u,v\ge1 \text{ and } n\ge1.
\end{equation}
Combining \eqref{eq:half-line-identification} and \eqref{eq:uchiyama-origin-bound} together with the fact from \eqref{H.asym} that $H_{K}(z)\ge c (z-K)$ for all $z\in E_K$ with some constant $c>0$, we have
\begin{equation}\label{eq:half-line-vs-origin}
p_S^{(K)}(n;x,y)\le q_n(x-K,y-K)  \le C(x-K+1)(y-K) n^{-3/2}\le C H_K(x+1)H_K(y)n^{-3/2}.
\end{equation}
Using this upper bound,
 we notice that
\begin{align*}
&\widehat{p}_S^{(K)}(1;x,y)-\widehat{p}_S^{(K)}(s;x,y)
=
\sum_{n\ge1}p_S^{(K)}(n;x,y)(1-s^n) \\
&\le
C\,H_K(x+1)\,H_K(y)\sum_{n\ge1}n^{-3/2}(1-s^n) \le
C\,H_K(x+1)\,H_K(y)\sqrt{1-s}.
\end{align*}
Moreover, by the asymptotic \eqref{eq:benchmark-local-asymptotic}, using Abel's theorem, we obtain
$$
\widehat{p}_S^{(K)}(1;x,y)-\widehat{p}_S^{(K)}(s;x,y)
=
\frac{\sqrt2}{\varsigma_\beta}\,H_K(x+1)\,H_K(y)\,\sqrt{1-s}
+
o(H_K(x+1)\,H_K(y)\sqrt{1-s}),
\quad s\uparrow1.
$$
\end{proof}

The next lemma records the exponential closeness between the transition probabilities of $S$ and $Y$ outside
the boundary layer $W_K$.
\begin{lemma}\label{lem:killed-kernel-comparison}
There exist constants $C, c\in (0,\infty)$, depending only on $\beta$, such that
for every $K\ge0$ and for every $x\in E_K$, we have
\begin{align}
\label{lem:weighted-perturbation}
&\frac{1}{H_K(x)}
\sum_{y\in E_K}|Q(x,y)-\bar Q(x,y)|\,H_K(y+1)
\le Ce^{-cx}.
\end{align}
\end{lemma}
\begin{proof}
Fix $K\ge0$. For $x,y\in E_K$, note that \begin{align*}
   &Q(x,y)=\mu_x(y-x)\quad \text{and} \quad
  \bar Q(x,y)=\nu(y-x).
\end{align*}
Recall from \eqref{H.asym} that there exists $C_0\ge1$,  such that
for all $x\in E_K$,
$$
C_0^{-1}(x-K+1)\le H_K(x)\le C_0(x-K+1).
$$
Therefore it suffices to prove that
\begin{equation}\label{eq:weight.sum}
\sum_{y\in E_K}|Q_K(x,y)-\bar Q_K(x,y)|\,(y-K+2)
\le
Ce^{-cx}(x-K+1),
\quad x\in E_K.
\end{equation}
Fix $x\in E_K$. 
We have
\begin{align*}
\sum_{y\in E_K}|Q_K(x,y)-\bar Q_K(x,y)|\,(y-K+2)&=
\sum_{y>K}|\mu_x(y-x)-\nu(y-x)|\,(y-K+2)\\
&\le
\sum_{j\in\mathbb Z}|\mu_x(j)-\nu(j)|\big(x-K+2+|j|\big)\\
&=
(x-K+2)\sum_{j\in\mathbb Z}|\mu_x(j)-\nu(j)|
+
\sum_{j\in\mathbb Z}|j|\,|\mu_x(j)-\nu(j)|.
\end{align*}
The first sum is bounded by
$C(x-K+1)e^{-cx}
$
by \eqref{lem:limiting-jump-law}.
For the second sum, splitting it into $|j|\le x$ and $|j|>x$, we have
\begin{align*}
\sum_{j\in\mathbb Z}|j|\,|\mu_x(j)-\nu(j)|
&\le
x\sum_{j\in\mathbb Z}|\mu_x(j)-\nu(j)|
+
\sum_{|j|>x}|j|\,\mu_x(j)
+
\sum_{|j|>x}|j|\,\nu(j).
\end{align*}
The first term is again bounded by $Cxe^{-cx}$ using
\eqref{lem:limiting-jump-law}. For the second term, 
using summation by parts and
\eqref{lem:uniform-tails}, we have
$$
\sum_{|j|>x}|j|\,\mu_x(j)
\le
\sum_{m>x}\mathbb P_x(|Y_1-Y_0|\ge m)
\le
Ce^{-cx}.
$$
The third term is bounded by $Ce^{-cx}$ as $\nu$ has Gaussian tails. 
Hence, \eqref{eq:weight.sum} is verified. This completes the
proof.
\end{proof}

\subsection{Excursion probabilities on the outside region}
\label{subsec:outside-kernels}
Throughout this subsection we keep $K\ge0$ fixed, and let
$$
Q_K:=(Q(x,y))_{x,y\in E_K}\quad
\text{and}\quad  \bar Q_K:=(\bar Q(x,y))_{x,y\in E_K}
$$
denote the transition probabilities of the Markov chains $Y$ and $S$
defined on $E_K$. In this subsection we compare the killed excursion probabilities of $Y$ $S$ outside the boundary region $W_K$.

For $x,y\in E_K$ and $n\ge1$, recall that $$
p_S^{(K)}(n;x,y)
:=
\mathbb P_x^S(S_1,\dots,S_{n-1}\in E_K, S_n=y),
$$ and define
$$
p_Y^{(K)}(n;x,y)
:=
\mathbb P_x^Y(Y_1,\dots,Y_{n-1}\in E_K, Y_n=y).
$$
We also use the convention
$$
p_Y^{(K)}(0;x,y)=p_S^{(K)}(0;x,y):=\mathbf 1_{\{x=y\}},
\quad x,y\in E_K.
$$
The following result is the Duhamel expansion for the difference of the two excursion transition probabilies.
\begin{lemma}[Duhamel formula for excursion probabilities]
\label{lem:duhamel-killed-kernels}
Let
$$
\Delta_K:=Q_K-\bar Q_K.
$$
Then, for every $m\ge1$,
and for all $x,y\in E_K$,
$$
p_Y^{(K)}(m;x,y)-p_S^{(K)}(m;x,y)
=
\sum_{j=0}^{m-1}\sum_{z,w\in E_K}
p_Y^{(K)}(j;x,z)\,\Delta_K(z,w)\,p_S^{(K)}(m-1-j;w,y).
$$
\end{lemma}

\begin{proof}
Note that
$$
p_Y^{(K)}(m;x,y)=(Q_K^m)(x,y),
\quad
p_S^{(K)}(m;x,y)=(\bar Q_K^m)(x,y).
$$
We notice that for $m\ge1$,
\begin{align*}
Q_K^{m}-\bar Q_K^{m}
=
Q_K^{m-1}(Q_K-\bar Q_K)+(Q_K^{m-1}-\bar Q_K^{m-1})\bar Q_K.
\end{align*}
Using the above identity and induction, we have
$$
Q_K^m-\bar Q_K^m
=
\sum_{j=0}^{m-1}Q_K^j(Q_K-\bar Q_K)\bar Q_K^{\,m-1-j}\quad
\text{for all } m\ge1.
$$
Taking the $(x,y)$-entry of both sides, we obtain the claimed formula.
\end{proof}

For $u, v\in W_K$, define
\begin{align}\label{AB.bound2}
    A_{K,u}:=\sum_{x\in E_K}Q(u,x)H_K(x+1),
\quad
B_{K,v}:=\sum_{y\in E_K}H_K(y)Q(y,v).
\end{align}
Note that ${A}_{K,u}, {B}_{K,v} \in (0,\infty)$ since $H$ has linear growth by \eqref{H.asym} while
$Q(u,x)=\mu_{u}(x-u), Q(y,v) =\mu_y(v-y)$
which have Gaussian tails in $x$ and $y$ respectively by \eqref{mu.tail}. For $u\in W_K$, $z\in E_K,\ n\ge0$ and $0\le s<1$, define
\begin{align*}
    f_Y^{(K)}(n;u,z)&:=\sum_{x\in E_K}Q(u,x)\,p_Y^{(K)}(n;x,z), \quad \widehat f_Y^{(K)}(s;u,z):=\sum_{n\ge0}f_Y^{(K)}(n;u,z)s^n,\\
    f_S^{(K)}(n;u,z)&:=\sum_{x\in E_K} Q(u,x)\,p_S^{(K)}(n;x,z), \quad \widehat f_S^{(K)}(s;u,z):=\sum_{n\ge0}f_S^{(K)}(n;u,z)s^n.
\end{align*}
\begin{lemma}
\label{lem:benchmark-mixed-excursion}
For every $K\ge0$, $u\in W_K$ and $z\in E_K$, we have
\begin{equation}\label{eq:fR-gf-asymptotic}
\widehat f_S^{(K)}(1;u,z)-\widehat f_S^{(K)}(s;u,z)
=
\frac{\sqrt2\,A_{K,u}}{\varsigma_\beta}\,H_K(z)\,\sqrt{1-s}
+
o\!\bigl(H_K(z)\sqrt{1-s}\bigr)
\quad \text{as }s\uparrow1.
\end{equation}
Moreover, there exists a constant $C<\infty$ such that for every $z\in E_K$ and every $n\ge1$,
\begin{equation}\label{eq:fR-upper}
f_S^{(K)}(n;u,z)\le C\,A_{K,u}\,H_K(z)\,n^{-3/2},
\end{equation}
and consequently, for every $0\le s<1$,
\begin{equation}\label{eq:fR-gf-bound}
0\le \widehat f_S^{(K)}(1;u,z)-\widehat f_S^{(K)}(s;u,z)
\le C\,A_{K,u}\,H_K(z)\,\sqrt{1-s}.
\end{equation}
\end{lemma}
\begin{proof}
Fix $K\ge0$, $u\in W_K$ and $z\in E_K$. By Lemma~\ref{prop:benchmark-first-return}, $p_S^{(K)}(n;x,z)\le C\,H_K(x+1)\,H_K(z)\,n^{-3/2}$ for each $x\in E_K, n\ge1$.
Hence
\begin{align*}
f_S^{(K)}(n;u,z)=\sum_{x\in E_K}Q(u,x)\,p_S^{(K)}(n;x,z)& \le
C\,H_K(z)\,n^{-3/2}\sum_{x\in E_K}Q(u,x)\,H_K(x+1)\\
&=
C\,A_{K,u}\,H_K(z)\,n^{-3/2},
\end{align*}
which proves \eqref{eq:fR-upper}. By the definition of $\widehat f_S^{(K)}$, we have
\begin{align*}
\widehat f_S^{(K)}(1;u,z)-\widehat f_S^{(K)}(s;u,z)
&=
\sum_{x\in E_K}Q(u,x)\sum_{n\ge0}p_S^{(K)}(n;x,z)(1-s^n)\\
&=
\sum_{x\in E_K}Q(u,x)\Bigl(\widehat p_S^{(K)}(1;x,z)-\widehat p_S^{(K)}(s;x,z)\Bigr).
\end{align*}
By Lemma~\ref{prop:benchmark-first-return}, for each fixed $x, z\in E_K$,
\begin{align*}
&\widehat p_S^{(K)}(1;x,z)-\widehat p_S^{(K)}(s;x,z)
=
\frac{\sqrt2}{\varsigma_\beta}\,H_K(x+1)\,H_K(z)\,\sqrt{1-s}
+
o\!\bigl(H_K(x+1)H_K(z)\sqrt{1-s}\bigr),\quad s\uparrow1,\\
&0\le \widehat p_S^{(K)}(1;x,z)-\widehat p_S^{(K)}(s;x,z)
\le
C\,H_K(x+1)\,H_K(z)\,\sqrt{1-s},
\quad 0\le s<1.
\end{align*}
Since
$
\sum_{x\in E_K}Q(u,x)\,H_K(x+1)=A_{K,u}<\infty,
$
the dominated convergence theorem implies
\begin{align*}
\widehat f_S^{(K)}(1;u,z)-\widehat f_S^{(K)}(s;u,z)
&=
\frac{\sqrt2}{\varsigma_\beta}\,H_K(z)\,\sqrt{1-s}
\sum_{x\in E_K}Q(u,x)\,H_K(x+1)
+
o\!\bigl(H_K(z)\sqrt{1-s}\bigr)\\
&=
\frac{\sqrt2\,A_{K,u}}{\varsigma_\beta}\,H_K(z)\,\sqrt{1-s}
+
o\!\bigl(H_K(z)\sqrt{1-s}\bigr),
\end{align*}
which is exactly \eqref{eq:fR-gf-asymptotic}. Finally,  using \eqref{eq:fR-upper}, we have 
\begin{align*}
0\le \widehat f_S^{(K)}(1;u,z)-\widehat f_S^{(K)}(s;u,z)
&=
\sum_{n\ge1}f_S^{(K)}(n;u,z)(1-s^n)\\
&\le
C\,A_{K,u}\,H_K(z)\sum_{n\ge1}n^{-3/2}(1-s^n)\le C\,A_{K,u}\,H_K(z)\,\sqrt{1-s}.
\end{align*}
This verifies \eqref{eq:fR-gf-bound}.
\end{proof}
\begin{lemma}\label{eq:f-upper}
There exists a constant $C<\infty$ and $K_0\in \N$ such that for all $K\ge K_0$,
\begin{equation*}
f_Y^{(K)}(n;u,z)\le C\,A_{K,u}\,H_K(z)\,n^{-3/2},
\quad z\in E_K, u\in W_K, \ n\ge1.
\end{equation*}
Consequently, for every fixed $z\in E_K$, $u\in W_K$, $K\ge K_0$,
\begin{equation*}
0\le \widehat f_Y^{(K)}(1;u,z)-\widehat f_Y^{(K)}(s;u,z)\le C A_{K,u}H_K(z)\,\sqrt{1-s}, \quad 0\le s< 1.
\end{equation*}
\end{lemma}
\begin{proof}
Using the Duhamel's formula in Lemma \ref{lem:duhamel-killed-kernels},
\begin{align*}
p_Y^{(K)}(m;x,z)-p_S^{(K)}(m;x,z)
&=
\sum_{j=0}^{m-1}\sum_{a,b\in E_K}
p_Y^{(K)}(j;x,a)\,\Delta_K(a,b)\,p_S^{(K)}(m-1-j;b,z).
\end{align*}
Multiply by $Q(u,x)$ and sum over $x\in E_K$, we get
\begin{align}
f_Y^{(K)}(m;u,z)
&=
f_S^{(K)}(m;u,z)
+\sum_{j=0}^{m-1}\sum_{a,b\in E_K}
f_Y^{(K)}(j;u,a)\,\Delta_K(a,b)\,p_S^{(K)}(m-1-j;b,z).
\label{eq:A-recursion}
\end{align}
By Lemma \ref{prop:benchmark-first-return},
$
p_S^{(K)}(m-1-j;b,z)\le C\,H_K(b+1)H_K(z)\,(m-j)^{-3/2}.
$
Therefore, using Lemma~\ref{lem:killed-kernel-comparison},
\begin{align*}
\sum_{b\in E_K}|\Delta_K(a,b)|\,p_S^{(K)}(m-1-j;b,z)
&\le
C H_K(z)(m-j)^{-3/2}
\sum_{b\in E_K}|\Delta_K(a,b)|\,H_K(b+1)\\
&\le
Ce^{-ca}H_K(a)H_K(z)(m-j)^{-3/2}.
\end{align*}
Also, by Lemma~\ref{lem:benchmark-mixed-excursion},
$
f_S^{(K)}(m;u,z)\le C {A}_{K,u}\,H_K(z)\,m^{-3/2}.
$
Taking absolute values in \eqref{eq:A-recursion} and using the previous bound,
we obtain
\begin{align}
f_Y^{(K)}(m;u,z)
&\le
C A_{K,u}H_K(z)m^{-3/2}
+
C H_K(z)\sum_{j=0}^{m-1}(m-j)^{-3/2}
\sum_{a\in E_K}f_Y^{(K)}(j;u,a)e^{-ca}H_K(a).
\label{eq:A-ineq}
\end{align}
Define
$$
M_m:=\sup_{1\le r\le m}\sup_{z\in E_K}\frac{f_Y^{(K)}(r;u,z)}{A_{K,u}H_K(z)\,r^{-3/2}},
\quad m\ge1.
$$
We have
\begin{align*}
\sum_{a\in E_K}f_Y^{(K)}(j;u,a)e^{-ca}H_K(a)
&\le
A_{K,u}M_{m-1}j^{-3/2}\sum_{a\in E_K}e^{-ca}H_K(a)^2.
\end{align*}
Since $a\in E_K$ implies $a\ge K+1$ and $H_K(a)\le C(a-K+1)$, the sum on the
right-hand side is bounded by $Ce^{-cK}$. Thus for $j\ge1$,
$$
\sum_{a\in E_K}f_Y^{(K)}(j;u,a)e^{-ca}H_K(a)\le Ce^{-cK}A_{K,u}M_{m-1}j^{-3/2}.
$$
The $j=0$ term is treated directly. Since
$f_Y^{(K)}(0;u,a)= Q(u,a)$ and $a\ge K+1$, we have
$$
\sum_{a\in E_K}f_Y^{(K)}(0;u,a)e^{-ca}H_K(a)\le Ce^{-cK}A_{K,u}.
$$
Substituting these bounds into \eqref{eq:A-ineq}, we obtain 
\begin{align*}
f_Y^{(K)}(m;u,z)
&\le
C {A}_{K,u}H_K(z)m^{-3/2}
+
C e^{-cK}A_{K,u}H_K(z)\Biggl[
m^{-3/2}
+
M_{m-1}\sum_{j=1}^{m-1}j^{-3/2}(m-j)^{-3/2}
\Biggr].
\end{align*}
Since
$
\sum_{j=1}^{m-1}j^{-3/2}(m-j)^{-3/2}\le C m^{-3/2},
$
we conclude that
$$
f_Y^{(K)}(m;u,z)\le
C\Bigl( 1+e^{-cK}(1+M_{m-1})\Bigr)A_{K,u}H_K(z)m^{-3/2}.
$$
Thus
$
M_m\le  C\big(1+ e^{-cK}(1+M_{m-1})\big).
$
Choosing $K_0$ sufficiently large such that $C e^{-cK}\le \frac14$ for all $K\ge K_0$, we obtain
$$
M_m\le C+\frac14(1+ M_{m-1}).
$$
A simple induction yields $\sup_m M_m<\infty$, and thus  
$$f_Y^{(K)}(n;u,z)\le C\,A_{K,u}\,H_K(z)\,n^{-3/2},
\quad z\in E_K, u\in W_K, \ n\ge1.$$
Moreover, we have
\begin{align*}
0\le \widehat f_Y^{(K)}(1;u,z)-\widehat f_Y^{(K)}(s;u,z)= \sum_{n\ge0}f_Y^{(K)}(n;u,z)(1-s^n)&\le
C\,A_{K,u}\,H_K(z)\sum_{n\ge1}n^{-3/2}(1-s^n)\\
&\le C \,A_{K,u}\,H_K(z)\sqrt{1-s}.
\end{align*}
This completes the proof.
\end{proof}

\begin{lemma}
\label{lem:mixed-perturbative-excursion}
There exists $K_0\in\mathbb N$ such that for every $K\ge K_0$ and every $u\in W_K$, there exists a function
$
C_{K,u}:E_K\to[0,\infty)
$
such that for every fixed $z\in E_K$, we have 
$$
\widehat f_Y^{(K)}(1;u,z)-\widehat f_Y^{(K)}(s;u,z)=C_{K,u}(z)\,H_K(z)\,\sqrt{1-s}
+
o\bigl(H_K(z)\,\sqrt{1-s}\bigr),
\quad s\uparrow1.
$$
In particular, there exists a constant $C\in (0,\infty)$ such that
$
C_{K,u}(z)\le C\,A_{K,u}$ for all
$z\in E_K, u\in W_K$ and $K\ge 0$.
\end{lemma}

\begin{proof} We divide the proof into several steps.

\textbf{Step 1: Expansion of $\widehat f_Y^{(K)}$ via Duhamel formula.}
Fix $K\ge K_0$, $u\in W_K$, and $z\in E_K$.
summing \eqref{eq:A-recursion} over $m\ge0$ yields
\begin{equation}\label{eq:gf-main}
\widehat f_Y^{(K)}(s;u,z)
=
\widehat f_S^{(K)}(s;u,z)
+
\sum_{a\in E_K}\widehat f_Y^{(K)}(s;u,a)\,\kappa(s;a,z),
\end{equation}
where for $a,z\in E_K$, we define
$$
\kappa(s;a,z):=
s\sum_{b\in E_K}\Delta_K(a,b)\,\widehat{p}_S^{(K)}(s;b,z).
$$
Subtracting \eqref{eq:gf-main} at $s$ from the corresponding identity at $1$, we obtain
\begin{align}\label{eq:U-eq}
\widehat f_Y^{(K)}(1;u,z)-\widehat f_Y^{(K)}(s;u,z)
&=
\widehat f_S^{(K)}(1;u,z)-\widehat f_S^{(K)}(s;u,z)
+
\sum_{a\in E_K}\big(\widehat f_Y^{(K)}(1;u,a)-\widehat f_Y^{(K)}(s;u,a) \big)\kappa(1;a,z)\\
\nonumber &+\sum_{a\in E_K}\widehat f_Y^{(K)}(s;u,a)\bigl(\kappa(1;a,z)-\kappa(s;a,z)\bigr).
\end{align}

We first record the basic bounds. By Lemma \ref{eq:f-upper},
\begin{equation}\label{eq:fYhat-bound}
\widehat f_Y^{(K)}(s;u,a)\le \widehat f_Y^{(K)}(1;u,a)\le C\,A_{K,u}\,H_K(a),
\quad a\in E_K,\ 0\le s<1.
\end{equation}
By Lemma~\ref{prop:benchmark-first-return}, 
$
\widehat{p}_S^{(K)}(1;b,z)\le C\,H_K(b+1)\,H_K(z)$,
for each $b,z\in E_K.$
Using this bound together with Lemma~\ref{lem:killed-kernel-comparison}, we get
\begin{align}\label{eq:beta1-bound}
|\kappa(1;a,z)|
\le
\sum_{b\in E_K}|\Delta_K(a,b)|\,\widehat{p}_S^{(K)}(1;b,z)
\le
C e^{-ca}H_K(a)\,H_K(z),
\quad a,z\in E_K.
\end{align}

Next, by Lemma~\ref{prop:benchmark-first-return}, 
$$
\widehat{p}_S^{(K)}(1;b,z)-\widehat{p}_S^{(K)}(s;b,z)
=
\frac{\sqrt2}{\varsigma_\beta}\,H_K(b+1)\,H_K(z)\,\sqrt{1-s}
+
o\bigl(H_K(b+1)\,H_K(z)\,\sqrt{1-s}\bigr)
$$
as $s\uparrow1$, and 
$
|\widehat{p}_S^{(K)}(1;b,z)-\widehat{p}_S^{(K)}(s;b,z)|
\le
C\,H_K(b+1)\,H_K(z)\,\sqrt{1-s}.
$
Also,
\begin{align*}
    \kappa(1;a,z)-\kappa(s;a,z)
&=
\sum_{b\in E_K}\Delta_K(a,b)\,\bigl(\widehat{p}_S^{(K)}(1;b,z)-\widehat{p}_S^{(K)}(s;b,z)\bigr)\\
&+
(1-s)\sum_{b\in E_K}\Delta_K(a,b)\,\widehat{p}_S^{(K)}(s;b,z).
\end{align*}
Since $\widehat{p}_S^{(K)}(s;b,z)\le \widehat{p}_S^{(K)}(1;b,z)\le C\,H_K(b+1)\,H_K(z)$, the second term is bounded by
$$
C(1-s)e^{-ca}H_K(a)H_K(z)
=
o\bigl(e^{-ca}H_K(a)H_K(z)\sqrt{1-s}\bigr)
$$
as $s\uparrow1$.
Set
$$
\delta_K(a):=
\frac{\sqrt2}{\varsigma_\beta}\sum_{b\in E_K}\Delta_K(a,b)\,H_K(b+1),
\quad a\in E_K,$$
and note that
$|\delta_K(a)|\le C e^{-ca}H_K(a)$,
for each $a\in E_K.$
Hence
\begin{equation}\label{eq:beta-expansion}
\kappa(1;a,z)-\kappa(s;a,z)
=
\delta_K(a)\,H_K(z)\,\sqrt{1-s}
+
o\bigl(e^{-ca}H_K(a)\,H_K(z)\,\sqrt{1-s}\bigr),
\end{equation}
as $s\uparrow1$, with the uniform bound
\begin{equation}\label{eq:beta-bound}
|\kappa(1;a,z)-\kappa(s;a,z)|
\le
C e^{-ca}H_K(a)\,H_K(z)\,\sqrt{1-s}.
\end{equation}

\textbf{Step 2: Functional equation for normalized generating functions.} 
Now fix $K\ge 0$ and $u\in W_K$. For $y\in E_K$ and $0\le s<1$, define
\begin{align*}
F_y^Y(s)&:=\frac{\widehat f_Y^{(K)}(1;u,y)-\widehat f_Y^{(K)}(s;u,y)}{H_K(y)\sqrt{1-s}},
\quad
F^S_y(s):=\frac{\widehat f_S^{(K)}(1;u,y)-\widehat f_S^{(K)}(s;u,y)}{H_K(y)\sqrt{1-s}},\\
 W_y(s)&:=\frac{\sum_{a\in E_K}\widehat f_Y^{(K)}(s;u,a)\,\bigl(\kappa(1;a,y)-\kappa(s;a,y)\bigr)}{H_K(y)\sqrt{1-s}}.
\end{align*}
We first notice that by Lemma \ref{eq:f-upper},
\begin{equation}\label{eq:F-bound}
|F^Y_z(s)|\le C\,A_{K,u},
\quad z\in E_K,\ 0\le s<1.
\end{equation}
Also, by Lemma~\ref{lem:benchmark-mixed-excursion}, for every $z\in E_K$,
\begin{equation}\label{eq:R-bound}
F^S_z(s)\to \frac{\sqrt2\,A_{K,u}}{\varsigma_\beta}
\quad\text{as }s\uparrow1\quad\text{and}\quad |F^S_z(s)|\le C\,A_{K,u},
\quad 0\le s<1.
\end{equation}
For $W_z(s)$, we notice that by \eqref{eq:fYhat-bound} and \eqref{eq:beta-bound},
\begin{align}\label{W.bound}
    |W_z(s)|\le C A_{K,u}.
\end{align}
Using \eqref{eq:beta-expansion} and
 applying dominated convergence in the sum over $a$, we obtain that for every fixed $z\in E_K$,
$$
W_z(s)\to L_{K,u}:=\sum_{a\in E_K}\widehat f_Y^{(K)}(1;u,a)\,\delta_K(a)\quad \text{as $s\uparrow1$}.
$$
This sum is finite since $|\delta_K(a)|\le C e^{-ca}H_K(a)$ and $\widehat f_Y^{(K)}(1;u,a)\le C A_{K,u}H_K(a)$ by Lemma \ref{eq:f-upper}.

Next, define
$$
G_z(a):=\frac{H_K(a)\,\kappa(1;a,z)}{H_K(z)},
\quad a,z\in E_K.
$$
Then \eqref{eq:U-eq} becomes
\begin{equation}\label{eq:F-eq}
F_z^Y(s)
=
F^S_z(s)
+
\sum_{a\in E_K}G_z(a)\,F_a^Y(s)
+
W_z(s).
\end{equation}
Moreover, by \eqref{eq:beta1-bound},
$
|G_z(a)|\le C e^{-ca}H_K(a)^2$, for each $a,z\in E_K$. Since $a\ge K+1$ on $E_K$ and $H_K(a)\le C(a-K+1)$, we have
$
\sup_{z\in E_K}\sum_{a\in E_K}|G_z(a)|\le Ce^{-cK}.
$
Therefore, for sufficiently large $K_0$, we have
\begin{equation}\label{G.bound}
    \sup_{z\in E_K}\sum_{a\in E_K}|G_z(a)|\le \frac12
\quad\text{for all }K\ge K_0.
\end{equation}

Define function $\psi_s: E_K\to \R$ by 
$$\psi_s(z):=F^S_z(s)+W_z(s),
\quad z\in E_K.
$$
Then \eqref{eq:F-eq} becomes
$$
F_z^Y(s)=\psi_s(z)+\sum_{a\in E_K}G_z(a)\,F_a^Y(s).
$$
By \eqref{eq:R-bound} and \eqref{W.bound}, we have
$
\sup_{z\in E_K}|\psi_s(z)|\le C A_{K,u}$,
for each $0\le s<1.$
Also, for each fixed $z\in E_K$,
\begin{align}\label{def.g}
    \psi_s(z)\to \widetilde \psi(z)\equiv \frac{\sqrt2\,A_{K,u}}{\varsigma_\beta}+L_{K,u} \quad\text{as }s\uparrow1.
\end{align}

\textbf{Step 3: Solution to the functional equation.} Let $\mathcal{G}$ be the bounded linear operator on $\ell^\infty(E_K)$ defined by
$$
(\mathcal{G}\varphi)(z):=\sum_{a\in E_K}G_z(a)\,\varphi(a).
$$
By \eqref{G.bound}, we have $\|\mathcal{G}\|\le \frac12$, and thus
$
(I-\mathcal{G})^{-1}=\sum_{m\ge0}\mathcal{G}^m
$
on $\ell^\infty(E_K)$. Therefore,
$$
F_z^Y(s)=\big((I-\mathcal{G})^{-1}\psi_s\big)(z)=\sum_{m\ge0}(\mathcal{G}^m \psi_s)(z).
$$
For each fixed $m$ and fixed $z$, the series defining $(\mathcal{G}^m \psi_s)(z)$ is absolutely summable. Using the dominated convergence theorem, we get
$
(\mathcal{G}^m \psi_s)(z)\to (\mathcal{G}^m \widetilde{\psi})(z)$
as $s\uparrow1,
$
where $\widetilde{\psi}$ is the constant function defined in \eqref{def.g}.
Moreover,
$
|(\mathcal{G}^m \psi_s)(z)|\le \|\mathcal{G}\|^m\sup_{y\in E_K}|\psi_s(y)|
\le C\,2^{-m},
$
uniformly in $s$. Therefore, by dominated convergence in $m$,
$$
F_z^Y(s)\to C_{K,u}(z):=\sum_{m\ge0}(\mathcal{G}^m \widetilde{\psi})(z)
\quad\text{as }s\uparrow1
$$
for every fixed $z\in E_K$. Recalling the definition of $F_z^Y(s)$, we conclude that
$$
\widehat f_Y^{(K)}(1;u,z)-\widehat f_Y^{(K)}(s;u,z)=C_{K,u}(z)\,H_K(z)\,\sqrt{1-s}
+
o\bigl(H_K(z)\,\sqrt{1-s}\bigr),
\quad s\uparrow1.
$$

Moreover, by \eqref{eq:F-bound},
$
|F_z^Y(s)|\le C\,A_{K,u}$ for all $z\in E_K$ and $0\le s<1$.
Passing to the limit $s\uparrow1$, we obtain that
$
C_{K,u}(z)\le C\,A_{K,u}$ for all
$z\in E_K$.
This completes the proof.
\end{proof}

\subsection{First-return probabilities of $Y$}
\label{subsec:first-return-kernels}

For $u,v\in W_K$ and $n\ge1$, the first return probabilities of $Y$ are defined by
$$
\mathcal{K}_Y^{(K)}(n;u,v)
:=
\mathbb P_u^Y(\sigma_{W_K}^{+}=n, Y_n=v) \quad \text{with}\quad\sigma_{W_K}^{+}:=\inf\{m\ge1:Y_m\in W_K\}.
$$

\begin{proposition}
\label{prop:first-return-comparison}
There exists a sufficiently large $K_0$ such that for every fixed $K\ge K_0$ and $u,v\in W_K$, we have 
\begin{align}
&
\Lambda_K(u,v):=\sum_{y\in E_K}C_{K,u}(y)\,H_K(y)\,Q(y,v)
<\infty \quad \text{and}\\
&\sum_{n=1}^{\infty}\mathcal{K}_Y^{(K)}(n;u,v)(1-s^n)=\Lambda_K(u,v)\sqrt{1-s}+o(\sqrt{1-s})\quad \text{ as } s\uparrow  1.
\end{align}
\end{proposition}

\begin{proof} By Lemma~\ref{lem:mixed-perturbative-excursion},
$
C_{K,u}(y)\le C\,A_{K,u}$ for all $y\in E_K, u\in W_K$.
Hence
$$
\Lambda_K(u,v)
=
\sum_{y\in E_K}C_{K,u}(y)\,H_K(y)\,Q(y,v)
\le
C\,A_{K,u}\sum_{y\in E_K}H_K(y)\,Q(y,v)
=
C\,A_{K,u}\,B_{K,v}<\infty.
$$

Fix $K\ge K_0$ and $u,v\in W_K$. For $n\ge2$, a first return of $Y$ to $W_K$ at time $n$ and location $v$, started from $u\in W_K$, must proceed as follows:
\begin{itemize}
\item the first step moves from $u\in W_K$ into some $x\in E_K$,
\item the process has an excursion inside $E_K$ of length $n-2$ from $x$ to some $y\in E_K$,
\item the final step moves from $y\in E_K$ into $v\in W_K$.
\end{itemize}
Hence, for all $u,v\in W_K$ and all $n\ge2$, we have
\begin{equation}
\label{lem:first-return-decomposition}
\mathcal{K}_Y^{(K)}(n;u,v)
=
\sum_{x\in E_K}\sum_{y\in E_K}
Q(u,x)\,p_Y^{(K)}(n-2;x,y)\,Q(y,v)=\sum_{y\in E_K} f_Y^{(K)}(n-2;u,y)\,Q(y,v).
\end{equation}
Note also that $\mathcal{K}_Y^{(K)}(1;u,v)=Q(u,v)$. Therefore,
\begin{align*}
\sum_{n=1}^{\infty}\mathcal{K}_Y^{(K)}(n;u,v)(1-s^n)
&=Q(u,v)(1-s) +
\sum_{n\ge2}\sum_{y\in E_K} f_Y^{(K)}(n-2,u,y)\,Q(y,v)\,(1-s^n) \\
&= Q(u,v)(1-s)+
\sum_{y\in E_K}Q(y,v)\sum_{m\ge0}f_Y^{(K)}(m;u,y)\,(1-s^{m+2}).
\end{align*}
Now write $
1-s^{m+2}=(1-s^m)+s^m(1-s^2)$. Therefore
\begin{align}
\sum_{n=1}^{\infty}\mathcal{K}_Y^{(K)}(n;u,v)(1-s^n)
&= Q(u,v)(1-s)+
\sum_{y\in E_K}Q(y,v)\sum_{m\ge0}f_Y^{(K)}(m;u,y)(1-s^m)
\notag\\
&\quad
+(1-s^2)\sum_{y\in E_K}Q(y,v)\sum_{m\ge0}f_Y^{(K)}(m;u,y)s^m.
\label{eq:K-split}
\end{align}

We treat the second term and the last term on the right-hand side separately. For the second term, Lemma~\ref{lem:mixed-perturbative-excursion} gives, for each fixed $y\in E_K$,
$$
\sum_{m\ge0}f_Y^{(K)}(m;u,y)(1-s^m)
=
C_{K,u}(y)H_K(y)\sqrt{1-s}
+
o(H_K(y)\sqrt{1-s}),
\quad s\uparrow1.
$$
Moreover, by Lemma \ref{eq:f-upper},
$
\sum_{m\ge0}f_Y^{(K)}(m;u,y)(1-s^m)
\le
C\,A_{K,u}\,H_K(y)\,\sqrt{1-s}
$ for each $y\in E_K$ and $m\ge1$. Since
$
\sum_{y\in E_K}H_K(y)\,Q(y,v)=B_{K,v}<\infty,
$
the dominated convergence theorem yields
\begin{align}
\sum_{y\in E_K}Q(y,v)\sum_{m\ge0}f_Y^{(K)}(m;u,y)(1-s^m)
&=
\sqrt{1-s}\sum_{y\in E_K}C_{K,u}(y)H_K(y)Q(y,v)
+
o(\sqrt{1-s})
\notag\\
&=
\Lambda_K(u,v)\,\sqrt{1-s}
+
o(\sqrt{1-s}).
\label{eq:first-main-term}
\end{align}

For the last term in \eqref{eq:K-split}, since
$
\sum_{m\ge0}f_Y^{(K)}(m;u,y)s^m\le \sum_{m\ge0}f_Y^{(K)}(m;u,y)\le C\,A_{K,u}\,H_K(y)
$ by Lemma \ref{eq:f-upper},
we obtain
\begin{align}\nonumber
0\le
(1-s^2)\sum_{y\in E_K}Q(y,v)\sum_{m\ge0}f_Y^{(K)}(m;u,y)s^m
&\le
C\,(1-s^2)\,A_{K,u}\sum_{y\in E_K}H_K(y)\,Q(y,v) \\
&\le
C\,(1-s)\,A_{K,u}\,B_{K,v}=o(\sqrt{1-s}). \label{eq:second-term-small}
\end{align}
Finally, combining \eqref{eq:K-split}, \eqref{eq:first-main-term}, and \eqref{eq:second-term-small}, we conclude that
$$
\sum_{n=1}^{\infty}\mathcal{K}_Y^{(K)}(n;u,v)\,(1-s^n)
=
\Lambda_K(u,v)\,\sqrt{1-s}
+
o(\sqrt{1-s}),
\quad s\uparrow1.
$$
This completes the proof.
\end{proof}

\subsection{Ruin probability}
\label{subsec:finite-state-renewal}

For $u,v \in W_K$ and $n \ge 1$, recall that
$$
\mathcal K_Y^{(K)}(n;u,v)
:=
\mathbb P_u^Y(\sigma_{W_K}^{+} = n,\ Y_n = v)
$$
are the first-return probabilities of the forward chain $Y$ to $W_K$. In this subsection, we convert the asymptotics of the first-return probabilities into the exact asymptotic behavior of the ruin probability.

By Lemma \ref{lem:recurrent}, the Markov chain $Y$ is irreducible and recurrent, and thus
$$
\mathbb P_u^Y(\sigma_{W_K}^{+}<\infty)=1.
$$
Equivalently, the matrix $P_K=(P_K(u,v))_{u,v\in W_K}$, with
$$
P_K(u,v) := 
\sum_{n \ge 1} \mathcal K_Y^{(K)}(n;u,v)=\mathbb P_u^Y(\sigma_{W_K}^{+}<\infty,\ Y_{\sigma_{W_K}^{+}}=v),
$$
is stochastic.

For each $n\ge 0$, define matrices $$\mathcal{K}_n:=(\mathcal{K}_Y^{(K)}(n,u,v))_{u,v\in W_K},\quad 
\mathcal{U}_n:=
(\mathbb P_u^Y(Y_n = v))_{u,v\in W_K}
$$
and their generating functions 
$$
\widehat{\mathcal K}(s):=\sum_{n\ge1}s^n\mathcal K_n,
\quad
\widehat{\mathcal{U}}(s):=\sum_{n\ge0}s^n \mathcal{U}_n,
\quad 0\le s<1.
$$
Let
\begin{align}\label{def.T}
    T_0 := 0,
\quad
T_{m+1} := \inf\{n > T_m : Y_n \in W_K\},
\quad m \ge 0,
\end{align}
be the successive return times of the chain $Y$ to the boundary layer $W_K$.
Note that $\zeta_n:=(Y_{T_n}, T_n)$ is a Markov renewal process with finite state space $W_K$. The following result follows  from the Markov renewal equation of $\zeta_n$. We however, present a direct proof for the sake of comprehensiveness.
\begin{lemma}[Matrix renewal decomposition]
\label{prop:matrix-renewal}
For $0<s<1$, we have
\begin{equation}\label{eq:gen-renewal}
\widehat{\mathcal{U}}(s)=I+\widehat{\mathcal K}(s)+\widehat{\mathcal K}(s)^2\cdots=\bigl(I-\widehat{\mathcal K}(s)\bigr)^{-1}.
\end{equation}
\end{lemma}

\begin{proof}
Recall that the stopping times $(T_n)_{n\ge0}$ are defined by \eqref{def.T}.
Fix $u,v \in W_K$ and $n \ge 0$. If $Y_0 = u$ and $Y_n = v$, then either
$n = 0$ and $u = v$, or else there exists a unique $r\ge1$ such that
$
T_r=n.
$
By the strong Markov property at the return times $T_1,T_2,\dots,T_{r-1}$, the corresponding excursion increments are independent and each has law given by the first-return probability matrix $\mathcal K_Y^{(K)}$. Therefore,
\begin{align*}
&\mathbb P_u^Y\bigl(T_r=n,\ Y_n=v,\ T_1=n_1,\dots,T_r-T_{r-1}=n_r,Y_{T_1}=u_1,\dots,Y_{T_{r-1}}=u_{r-1}\bigr)\\
&\quad=
\mathcal K_Y^{(K)}(n_1;u,u_1)\,
\mathcal K_Y^{(K)}(n_2;u_1,u_2)\cdots
\mathcal K_Y^{(K)}(n_r;u_{r-1},v).
\end{align*}
Summing over all $r\ge1$, all compositions $n_1+\cdots+n_r=n$ with $n_i\ge1$, and all intermediate states $u_1,\dots,u_{r-1}\in W_K$, we obtain
$$\mathbf 1_{\{n=0,\ u=v\}}
+
\sum_{r=1}^{n}
\ \sum_{\substack{n_1,\dots,n_r\ge 1\\ n_1+\cdots+n_r=n}}
\ \sum_{u_1,\dots,u_{r-1}\in W_K}
\mathcal K_Y^{(K)}(n_1;u,u_1)\,
\mathcal K_Y^{(K)}(n_2;u_1,u_2)\cdots
\mathcal K_Y^{(K)}(n_r;u_{r-1},v),$$
where term $\mathbf 1_{\{n=0,\ u=v\}}$ accounts for the trivial case $n=0$. Taking the matrix generating function, we obtain the result of the lemma. 
\end{proof}

Since $P_K$ is a positive stochastic matrix on the finite set $W_K$, the
Perron--Frobenius eigenvalue $1$ is simple, with right eigenvector ${\bf 1}$
and left eigenvector ${\boldsymbol \pi}_K$. Set
$$
\gamma_K:={\boldsymbol \pi}_K\Lambda_K{\bf 1}.
$$ The next proposition gives the asymptotic order of the ruin probability.
\begin{proposition}
\label{thm:ruin-order}
For a fixed sufficiently
large $K$, we have 
$$
r_n=\mathbb P_0(Y_{n-1}=0)
=
\mathcal{U}_{n-1}(0,0)
\sim
\frac{{\boldsymbol \pi}_K(0)}{\sqrt{\pi}\,\gamma_K}\,n^{-1/2}.
$$
\end{proposition}

\begin{proof}
Recall from the matrix renewal decomposition in Lemma \ref{prop:matrix-renewal}, 
\begin{equation*}
\widehat{\mathcal{U}}(s)=\bigl(I-\widehat{\mathcal K}(s)\bigr)^{-1}.
\end{equation*}
We first identify the singular behavior of $\widehat{\mathcal K}(s)$ near
$s=1$. Note that $P_K=\widehat{\mathcal K}(1)$. By Proposition \ref{prop:first-return-comparison},
$$
P_K(u,v)-\widehat{\mathcal K}(s)(u,v)
=
\Lambda_K(u,v)\,\sqrt{1-s}
+
o(\sqrt{1-s}),
\quad s\uparrow1.
$$
Equivalently,
$$
\widehat{\mathcal K}(s)
=
P_K-\Lambda_K\,\sqrt{1-s}
+
o(\sqrt{1-s}),
\quad s\uparrow1.
$$
Let $\Pi_K:={\bf 1}\,{\boldsymbol \pi}_K$ be the rank-one spectral projection of the stochastic matrix $P_K$
corresponding to the eigenvalue $1$. Let $\lambda(s)$ denote the
eigenvalue of $\widehat{\mathcal K}(s)$ converging to $1$ as $s\uparrow1$. Then
\begin{align}
    \label{lambda.idn}
    \lambda(s)
=
1-{\boldsymbol \pi}_K\Lambda_K{\bf 1}\,\sqrt{1-s}
+
o(\sqrt{1-s})
=
1-\gamma_K\,\sqrt{1-s}
+
o(\sqrt{1-s}),
\quad s\uparrow1.
\end{align}
By continuity of the eigenvalues of a finite matrix, there exists $\delta>0$ such
that all eigenvalues of $\widehat{\mathcal K}(s)$ other than $\lambda(s)$ have
modulus at most $1-\delta$ for all $s$ sufficiently close to $1$. Let $\Pi(s)$ be
the spectral projection corresponding to $\lambda(s)$, and write
$$
\widehat{\mathcal K}(s)=\lambda(s)\Pi(s)+N(s),
$$
where
$
\Pi(s)N(s)=N(s)\Pi(s)=0$ and $\rho(N(s))\le 1-\delta.
$
Then
$$
I-\widehat{\mathcal K}(s)
=
(1-\lambda(s))\Pi(s)+(I-N(s))(I-\Pi(s)),
$$
and therefore
$$
\bigl(I-\widehat{\mathcal K}(s)\bigr)^{-1}
=
\frac{1}{1-\lambda(s)}\,\Pi(s)
+
(I-N(s))^{-1}(I-\Pi(s)).
$$
The second term is bounded as $s\uparrow1$ because $\rho(N(s))\le 1-\delta$.
Moreover, since $\widehat{\mathcal K}(s)\to P_K$, we have $\Pi(s)\to \Pi_K$ as
$s\uparrow1$. Hence
$$
\bigl(I-\widehat{\mathcal K}(s)\bigr)^{-1}
=
\frac{1}{1-\lambda(s)}\,\Pi_K+O(1),
\quad s\uparrow1.
$$
Using the expansion \eqref{lambda.idn}, we obtain
$$
\widehat{\mathcal{U}}(s)
=
\frac{1}{\gamma_K}\,(1-s)^{-1/2}\,\Pi_K+O(1),
\quad s\uparrow1.
$$
Notice that the coefficients of $(1-s)^{-1/2}=\sum_{n=0}^{\infty} a_n s^n$ satisfy
$$
a_n=\frac{\binom{2n}{n}}{4^n}\sim \frac{1}{\sqrt{\pi}}\,n^{-1/2},
\quad n\to\infty.
$$

On the other hand, using the fact that $Q(z,0)$ is decreasing in $z$, one can show by stochastic domination that $\mathcal{U}_{n}(0,0)=\mathbb P_0(Y_{n}=0)$ is {non-increasing in $n$}. Therefore, by the Karamata’s Tauberian theorem for nonnegative power
series (see Corollary 1.7.3 in \cite{BGT1987}), we obtain
$$
\mathcal{U}_n(0,0)
\sim
\frac{{\boldsymbol \pi}_K(0)}{\gamma_K}\cdot\frac{1}{\sqrt{\pi}}\,n^{-1/2}.
$$
Since, by Lemma~\ref{prop:auxiliary-transition},
$
r_n=\mathbb P_0(Y_{n-1}=0)=\mathcal{U}_{n-1}(0,0),
$
we conclude that
$$
r_n
\sim
\frac{{\boldsymbol \pi}_K(0)}{\sqrt{\pi}\,\gamma_K}\,n^{-1/2}.
$$
\end{proof}

\subsection{Interval estimates for the killed chain}\label{subsec:interval-killed}

For $j\ge 1$, let
$$
A_j:=\{1,\dots,j-1\}.
$$
In this subsection we prove an estimate for the probability of the joint event
that $Y_n$ takes values in $A_j$ and $\sigma_0^+>n$. This estimate will be used
in Section~\ref{sec:gen.ruin} to bound generalized ruin probabilities.

For each $a\in \Z_+$, let $$\sigma_a:=\inf\{n\ge 0:Y_n=a\},\quad \sigma_a^+:=\inf\{n\ge 1:Y_n=a\}, \quad\widetilde{\sigma}_a:=\inf\{n\ge 0:Y_n\ge a\}.$$
For $L\ge 2$, let
$$
T_L:
=\sigma_0\wedge \widetilde{\sigma}_L=
\inf\{n\ge 0:Y_n=0 \text{ or } Y_n\ge L\}.
$$
\begin{lemma}\label{lem:box-green-j2}
There exists a constant $C<\infty$ such that for all integers $L\ge 2$ and
$1\le j\le L/2$,
$$
\sup_{1\le y<L}
\E_y\!\left[\sum_{k=0}^{T_L-1}\mathbf 1_{\{Y_k\in A_j\}}\right]
\le C j^2.
$$
\end{lemma}

\begin{proof}
For $1\le a<L$, define
$$
G_L(y,a):=
\E_y\!\left[\sum_{k=0}^{T_L-1}\mathbf 1_{\{Y_k=a\}}\right].
$$
Then
$$
\E_y\!\left[\sum_{k=0}^{T_L-1}\mathbf 1_{\{Y_k\in A_j\}}\right]
=
\sum_{a=1}^{j-1}G_L(y,a).
$$
By the strong Markov property at time $\sigma_a$, we note that
$
G_L(y,a)=\P_y(\sigma_a<T_L)\,G_L(a,a)\le G_L(a,a).
$
Hence, to complete the proof, it is sufficient to show that $$
G_L(a,a)\le C a
\quad\text{for all }1\le a<L/2.
$$

Let
$
p_{a,L}:=\P_a(T_L<\sigma_a^+).
$
The number of visits to $a$ before $T_L$ is geometric with success
parameter $p_{a,L}$, and therefore
$
G_L(a,a)=1/{p_{a,L}}.
$
Hence, it is enough to prove that
$$
p_{a,L}\ge \frac{c}{a}
\quad\text{for all }1\le a<L/2,
$$
with a constant $c>0$ independent of $L$. On the other hand,
since $p_{a,L}\ge \P_a(\sigma_0<\sigma_a^+)$, it suffices to show that
$$
p_a:=\P_a(\sigma_0<\sigma_a^+)\ge \frac{c}{a}
\quad\text{for all }a\ge 1.
$$

By Lemma~\ref{lem:step_distr}, the law of the increment $Y_1-x$ converges
exponentially fast to the centered law $\nu$. Since $\nu$ is non-degenerate,
there exist an integer $r\ge 1$, a constant $\eta>0$, and $a_0<\infty$ such that
$$
\P_x(Y_1\le x-r)\ge \eta
\quad\text{for all }x\ge a_0.
$$
Fix $a\ge a_0+r$, and let $\widetilde{\sigma}_a:=\inf\{n\ge 0:Y_n\ge a\}$. For any $z\le a-r$,
consider the stopping time $\tau:=\sigma_0\wedge \widetilde{\sigma}_a$. Since $h$ is non-negative
and harmonic for the chain killed at $0$, the process $h(Y_{n\wedge\tau})$ is a
non-negative martingale. Hence
$$
h(z)=\E_z[h(Y_\tau)]
\ge
\E_z[h(Y_{\widetilde{\sigma}_a});\widetilde{\sigma}_a<\sigma_0].
$$
By Lemma~\ref{lem:harmonic}, $h(x)=x+O(1)$ as $x\to\infty$. Therefore there exists
$C_1<\infty$ such that $|h(x)-x|\le C_1$ for all $x\ge 1$, and hence
$
h(Y_{\widetilde{\sigma}_a})\ge a-C_1$
on $\{\widetilde{\sigma}_a<\sigma_0\}.
$
It follows that
$
h(z)\ge (a-C_1)\,\P_z(\widetilde{\sigma}_a<\sigma_0),
$
and thus
$$
\P_z(\sigma_0<\widetilde{\sigma}_a)\ge 1-\frac{h(z)}{a-C_1}.
$$
Since $z\le a-r$ and $h(z)\le z+C_1\le a-r+C_1$, we obtain
$$
\P_z(\sigma_0<\widetilde{\sigma}_a)\ge \frac{r-2C_1}{a-C_1}.
$$
Choose $r>2C_1$. Then there exists $c_1>0$ such that
$\inf_{z\le a-r}\P_z(\sigma_0<\widetilde{\sigma}_a)\ge c_1/{a}$ for $a\ge a_0+r$. Therefore
$$
p_a
\ge
\P_a(Y_1\le a-r)\,
\inf_{z\le a-r}\P_z(\sigma_0<\widetilde{\sigma}_a)
\ge
\frac{\eta c_1}{a},
\quad \text{for }a\ge a_0+r.
$$
For the finitely many values $1\le a<a_0+r$, irreducibility and accessibility of $0$ from $a$ before returning to $a$ imply that $p_a>0$. Hence
$
c_0:=\min_{1\le a<a_0+r} a\,p_a>0.
$
Therefore,
$$
p_a\ge \frac{c}{a}
\quad\text{for all }a\ge 1
\quad\text{with }
c:=\min\{c_0,\eta c_1\}>0.
$$
It follows that for $1\le a<L/2$, 
$$
G_L(a,a)\le \frac{a}{c},
$$
and hence $G_L(y,a)\le a/c$. Summing over $a\in A_j$, we have 
$$
\sup_{1\le y<L}
\E_y\!\left[\sum_{k=0}^{T_L-1}\mathbf 1_{\{Y_k\in A_j\}}\right]
\le
\frac{1}{c}\sum_{a=1}^{j-1}a
\le C j^2.
$$
This proves the lemma.
\end{proof}

\begin{lemma}\label{lem:survival-x-over-sqrtn}
There exists a constant $C<\infty$ such that for all $x\ge 1$ and all $n\ge 1$,
$$
\P_x(\sigma_0>n)\le C\,\frac{x}{\sqrt{n}}.
$$
\end{lemma}

\begin{proof}
Recall that for $x\ge 0$,
$$
h_N(x):=N\,\P_x(\widetilde \sigma_N<\sigma_0).$$
From the proof of Lemma~\ref{lem:harmonic}, there exists $C_0<\infty$ such that
\begin{equation}\label{eq:HN-profile-proof}
|h_N(x)-x|\le C_0
\quad\text{for all }N\ge 2,\ 1\le x<N.
\end{equation}

We first prove the mean exit-time bound
\begin{equation}\label{eq:mean-exit-xN}
\E_x[T_N]\le C_1\,xN
\quad\text{for all }N\ge 2,\ 1\le x\le N/2.
\end{equation}
Notice that
$$
\E_x[T_N]=\sum_{a=1}^{N-1}G_N(x,a) \quad\text{with}\quad G_N(x,a):=
\E_x\!\left[\sum_{k=0}^{T_N-1}\mathbf 1_{\{Y_k=a\}}\right]\quad\text{for } 1\le a<N.
$$
We claim that there exists $C_2<\infty$ such that
\begin{align}\label{eq:GN-upper-low}
G_N(x,a)&\le C_2\min\{x,a\}
\quad\text{for all }1\le a\le N/2,\ 1\le x<N,\quad\text{and}\\
\label{eq:GN-upper-high}
G_N(x,a)&\le C_2\,\frac{x}{N}\,(N-a+1)
\quad\text{for all }N/2<a<N,\ 1\le x\le N/2.
\end{align}

We start with the case $a\le N/2$.
By the strong Markov property at time $\sigma_a$,
$$
G_N(x,a)=\P_x(\sigma_a<T_N)\,G_N(a,a).
$$
Let
$
p_{a,N}:=\P_a(T_N<\sigma_a^+).
$
Then the number of visits to $a$ before $T_N$ is geometric with success parameter
$p_{a,N}$, so
$$
G_N(a,a)=\frac{1}{p_{a,N}}.
$$
As in the proof of Lemma~\ref{lem:box-green-j2}, we have
$
p_{a,N}\ge \P_a(\sigma_0<\sigma_a^+)\ge c/a
$
for all $a\ge 1$, with a constant $c>0$ independent of $N$. Hence
$$
G_N(a,a)\le \frac{a}{c}.
$$
If $x\ge a$, this gives $G_N(x,a)\le G_N(a,a)\le  C a=C\min\{x,a\}$.
If $x<a$, then on the event $\{\sigma_a<T_N\}$ the chain must reach level $a$
before hitting $0$, and therefore
$$
\P_x(\sigma_a<T_N)\le \P_x(\widetilde \sigma_a<\sigma_0)=\frac{h_a(x)}{a}.
$$
By \eqref{eq:HN-profile-proof},
$h_a(x)/{a}\le (x+C_0)/a\le C{x}/{a}$
for each $1\le x<a.
$
Thus
$$
G_N(x,a)\le \P_x(\sigma_a<T_N)\,G_N(a,a)\le C\,\frac{x}{a}\cdot a=Cx.
$$
This proves \eqref{eq:GN-upper-low}.

We now consider the case $N/2<a<N$ and $1\le x\le N/2$.
Again,
\begin{align}\label{GN.inq}
    G_N(x,a)\le \P_x(\sigma_a<T_N)\,G_N(a,a).
\end{align}
As before,
$$
\P_x(\sigma_a<T_N)\le \P_x(\widetilde\sigma_a<\sigma_0)=\frac{h_a(x)}{a}\le C\,\frac{x}{a}\le C\,\frac{x}{N}.
$$
It remains to bound $G_N(a,a)$. 

We show that
\begin{align}\label{p.aN_bound}
    p_{a,N}\ge \frac{c'}{N-a+1}
\quad\text{for all }N/2<a<N.\end{align}
By Lemma~\ref{lem:step_distr}, the increment law of $Y_1-Y_0$ converges
exponentially fast to the centered non-degenerate law $\nu$. Hence there exist an
integer $r\ge 1$, a constant $\eta>0$, and $z_0<\infty$ such that
$$
\P_z(Y_1\ge z+r)\ge \eta
\quad\text{for all }z\ge z_0.
$$
We choose $r>2C_0+1$. First, consider the finitely many values of $N$ with $N\le 2z_0+2r$.
Since the state space $\{1,\dots,N-1\}$ is finite and $0$ and $[N,\infty)$ are both
accessible from every interior state, there exists a constant $C_3<\infty$ such that
$$
G_N(a,a)\le C_3(N-a+1)
\quad\text{for all }N\le 2z_0+2r,\ N/2<a<N.
$$
Thus it remains to treat the case $N>2z_0+2r$. Then $a>N/2$ implies $a\ge z_0+r$.
If $a\in[N-r,N)$, then
$$
p_{a,N}\ge \P_a(Y_1\ge N)\ge \eta\ge \eta\,\frac{1}{N-a+1}.
$$
Now assume $N/2<a<N-r$. Starting from $a$, on the event $\{Y_1\ge a+r\}$ the chain
jumps to some state $z\ge a+r$. If $z\ge N$, then $T_N<\widetilde \sigma_a$ already.
If $a+r\le z<N$, let
$$
\tau:=\widetilde \sigma_N\wedge \inf\{k\ge 0:Y_k\le a\}.
$$
Since $0\le h_N\le N$, the process $h_N(Y_{m\wedge\tau})$ is a bounded martingale, and
therefore
$$
h_N(z)=\E_z[h_N(Y_\tau)].
$$
On the event $\{\widetilde\sigma_N<\inf\{k\ge 0:Y_k\le a\}\}$, we have $Y_\tau\ge N$ and hence
$h_N(Y_\tau)=N$. Otherwise $Y_\tau\le a$, and by \eqref{eq:HN-profile-proof},
$h_N(Y_\tau)\le a+C_0$. Thus
$$
h_N(z)\le (a+C_0)+(N-a-C_0)\,
\P_z\!\left(\widetilde\sigma_N<\inf\{k\ge 0:Y_k\le a\}\right).
$$
Since $z\ge a+r$ and \eqref{eq:HN-profile-proof} gives
$h_N(z)\ge z-C_0\ge a+r-C_0$, it follows that
$$
\P_z\!\left(\widetilde \sigma_N<\inf\{k\ge 0:Y_k\le a\}\right)
\ge \frac{r-2C_0}{N-a-C_0}.
$$
Because $a<N-r$ and $r>2C_0+1$, the denominator is positive, and therefore
$$
\P_z\!\left(\widetilde \sigma_N<\inf\{k\ge 0:Y_k\le a\}\right)\ge \frac{c_3}{N-a+1}
$$
for some $c_3>0$ independent of $N$ and $a$.
Combining with $\P_a(Y_1\ge a+r)\ge \eta$, we obtain \eqref{p.aN_bound}.
Thus
$$
G_N(a,a)=\frac{1}{p_{a,N}}\le C\,(N-a+1).
$$
Together with \eqref{GN.inq}, this implies \eqref{eq:GN-upper-high}.

Combining \eqref{eq:GN-upper-low} and \eqref{eq:GN-upper-high}, we obtain that for $1\le x\le N/2$,
\begin{align*}
\E_x[T_N]
&=
\sum_{a=1}^{N-1}G_N(x,a)=
\sum_{a=1}^{\lfloor N/2\rfloor}G_N(x,a)
+
\sum_{a=\lfloor N/2\rfloor+1}^{N-1}G_N(x,a)\\
&\le
C_2\sum_{a=1}^{\lfloor N/2\rfloor}\min\{x,a\}
+
C_2\frac{x}{N}\sum_{a=\lfloor N/2\rfloor+1}^{N-1}(N-a+1)\le Cx N.
\end{align*}
This proves \eqref{eq:mean-exit-xN}.

We now prove the claim of the lemma. Fix $x\ge 1$ and $n\ge 1$.
If $x>\sqrt{n+1}$, then the claimed bound is immediate after increasing the constant.
Thus we may assume $x\le \sqrt{n+1}$.
Set
$
N:=2\lceil\sqrt{n+1}\rceil.
$
Then $x\le N/2$. Since
$
\{\sigma_0>n\}\subseteq \{\widetilde \sigma_N<\sigma_0\}\cup \{T_N>n\},
$
we have
$$
\P_x(\sigma_0>n)\le \P_x(\widetilde \sigma_N<\sigma_0)+\P_x(T_N>n).
$$
By definition of $h_N$ and \eqref{eq:HN-profile-proof},
$$
\P_x(\widetilde \sigma_N<\sigma_0)=\frac{h_N(x)}{N}\le \frac{x+C_0}{N}\le C_4\,\frac{x}{N}.
$$
Also, by Markov's inequality and \eqref{eq:mean-exit-xN},
$$
\P_x(T_N>n)\le \frac{\E_x[T_N]}{n}\le C_1\,\frac{xN}{n}.
$$
Since $N:=2\lceil\sqrt{n+1}\rceil$, the last two bounds imply
$$
\P_x(\sigma_0>n)\le C\,\frac{x}{\sqrt{n+1}}.
$$
This proves the lemma.
\end{proof}

Recall that
$\sigma_0^+:=\inf\{m\ge 1:Y_m=0\}$.

\begin{lemma}\label{lem:interval-killed}
There exists a constant $C<\infty$ such that for all integers $M\ge1$ and
all $j\ge1$,
$$
\sum_{m=M}^{2M}
\P_0\!\left(Y_m\in A_j,\sigma_0^+>m\right)
\le
C\,j^2\,(M+1)^{-1/2}
.
$$
\end{lemma}

\begin{proof}
We first prove that there exists a positive constant $C_1<\infty$ such that
\begin{equation}\label{eq:half-line-green-Aj}
\sup_{y\ge1}
\E_y\!\left[
\sum_{k=0}^{\sigma_0-1}\mathbf 1_{\{Y_k\in A_j\}}
\right]
\le C_1j^2
\quad\text{for all } j\ge1.
\end{equation}
The case $j=1$ is trivial. Let $j\ge2$ and fix
$y\ge1$. Choose $L$ sufficiently large so that $L>y$ and $j\le L/2$. By
Lemma~\ref{lem:box-green-j2},
$$
\E_y\!\left[
\sum_{k=0}^{T_L-1}\mathbf 1_{\{Y_k\in A_j\}}
\right]
\le C_1j^2.
$$
Since the Markov chain $(Y_k)_{k\ge 0}$ is recurrent by Lemma~\ref{lem:recurrent}, we have $\sigma_0<\infty$ a.s. under
$\P_y$. Hence $T_L\uparrow\sigma_0$ a.s. as $L\to\infty$. Letting $L\to\infty$
and using monotone convergence, we get
$$
\E_y\!\left[
\sum_{k=0}^{\sigma_0-1}\mathbf 1_{\{Y_k\in A_j\}}
\right]
\le C_1 j^2.
$$
Taking the supremum over $y\ge1$, we obtain \eqref{eq:half-line-green-Aj}.

The result of the lemma is trivial when $M=1$. We assume that $M\ge2$ and set $h:=\lfloor M/2\rfloor$.
For $m\in\{M,\dots,2M\}$, by the Markov property at time $h$, 
\begin{align*}
\P_0(Y_m\in A_j,\sigma_0^+>m)
&=
\E_0\!\left[
\mathbf 1_{\{\sigma_0^+>h\}}
\P_{Y_h}(Y_{m-h}\in A_j,\sigma_0>m-h)
\right].
\end{align*}
On the event $\{\sigma_0^+>h\}$, we have $Y_h\ge1$. Therefore, summing over
$m=M,\dots,2M$ and using \eqref{eq:half-line-green-Aj}, we get
\begin{align}\label{est1}
\sum_{m=M}^{2M}
\P_0(Y_m\in A_j,\sigma_0^+>m)
&\le
\P_0(\sigma_0^+>h)
\sup_{y\ge1}
\sum_{s\ge0}
\P_y(Y_s\in A_j,\sigma_0>s)\le
C_1 j^2\,\P_0(\sigma_0^+>h).
\end{align}

It remains to bound $\P_0(\sigma_0^+>h)$. If $h\le1$, this probability is at
most $1$. If $h\ge2$, then by the Markov property at time $1$,
$$
\P_0(\sigma_0^+>h)
=
\sum_{x\ge1}\P_0(Y_1=x)\P_x(\sigma_0>h-1).
$$
By Lemma~\ref{lem:survival-x-over-sqrtn},
$$
\P_x(\sigma_0>h-1)\le C_2\frac{x}{\sqrt h}.
$$
Since $Y_1$ has finite first moment under $\P_0$ by Lemma~\ref{lem:step_distr},
we obtain
\begin{align}\label{est2}
   \P_0(\sigma_0^+>h)\le C_3 (h+1)^{-1/2}. 
\end{align}
Combining \eqref{est1} and \eqref{est2}, we get
$$
\sum_{m=M}^{2M}
\P_0(Y_m\in A_j,\sigma_0^+>m)
\le
Cj^2(M+1)^{-1/2}.
$$
This completes the proof.
\end{proof}

\subsection{Generalized ruin probability}\label{sec:gen.ruin}
This subsection, we extend the one-dimensional ruin probability estimate to the probability of hitting the endpoint before the $j$-th return to $0$.

\begin{proposition}
\label{prop:gen-ruin-upper}
For $n\ge 1$ and $j\ge 1$, let $\tau_0^{(j)}$ denote the time of the $j$-th
return to $0$ for the TSAW $\widetilde \X=(\widetilde{X}_k)_{k\ge 0}$ on $\{0,1,\dots,n\}$, namely
$$
\tau_0^{(1)}:=\tau_0^+,
\quad
\tau_0^{(j+1)}:=\inf\{k>\tau_0^{(j)}: \widetilde X_k=0\}\quad
\text{for } j\ge 1,
$$
and define the generalized ruin probability
$$
r_n^{(j)}:=\P(\tau_n<\tau_0^{(j)}).
$$
Then there exists a constant $C<\infty$ such that for all $n\ge 1$ and all
$j\ge 1$,
$$
r_n^{(j)}\le C\Bigl(1\wedge \frac{j}{\sqrt n}\Bigr).
$$
\end{proposition}

\begin{proof}
When $j^2>n/2$, the result of the lemma is trivial. Thus, in the rest of the proof, we assume that
\begin{equation}\label{eq:j2-small}
1\le j^2\le n/2. 
\end{equation}
Recall that
$
B(1,n)
$
is the number of backward jumps of the TSAW $(\widetilde{X}_k)_{k\ge 0}$ from $1$ to $0$ before the first hit of $n$.
Since each jump from $1$ to $0$ creates exactly one further visit to $0$ after
time $0$, the event $\{\tau_n<\tau_0^{(j)}\}$ is equivalent to the event that
there are at most $j-1$ such backward jumps before $\tau_n$. 
On the other hand, by Lemma~\ref{prop:auxiliary-transition},
$
(B(n,n),B(n-1,n),\dots,B(1,n))
\stackrel d=
(Y_0,Y_1,\dots,Y_{n-1}),
$
and in particular,
$
B(1,n)\stackrel d=Y_{n-1}$ under
$\P_0$.
Therefore,
\begin{equation}\label{eq:gen-ruin-Y}
r_n^{(j)}=\P\bigl(B(1,n)\le j-1\bigr)=\P_0(Y_{n-1}\le j-1)=\sum_{a=0}^{j-1}\P_0(Y_{n-1}=a).
\end{equation}

By Proposition~\ref{thm:ruin-order}, there exists a constant $C_0<\infty$ such
that
\begin{equation}\label{eq:r_n-upper}
r_n=\P_0(Y_{n-1}=0)\le C_0 n^{-1/2}
\quad\text{for all }n\ge 1.
\end{equation}

We next prove that there exists a constant $C_1<\infty$ such that, for all $m\ge1$,
\begin{equation}\label{eq:sigma0-plus-upper-direct}
\P_0(\sigma_0^+>m)\le C_1(m+1)^{-1/2}.
\end{equation}
The case $m=1$ is trivial. For $m\ge2$, by the Markov property at time $1$,
$$
\P_0(\sigma_0^+>m)
=
\sum_{x\ge1}\P_0(Y_1=x)\P_x(\sigma_0>m-1).
$$
Note that, by Lemma~\ref{lem:survival-x-over-sqrtn}, we have
$
\P_x(\sigma_0>m-1)\le C{x}/{\sqrt m}
$ for some positive constant $C<\infty$. Since $Y_1$ has finite first moment under $\P_0$ by Lemma~\ref{lem:step_distr},
we thus obtain
$$
\P_0(\sigma_0^+>m)
\le
C m^{-1/2}\sum_{x\ge1}x\,\P_0(Y_1=x)
\le C_1 m^{-1/2}.
$$
This verifies \eqref{eq:sigma0-plus-upper-direct}.

By the strong Markov property, for $a\ge 1$,
\begin{align}
\P_0(Y_{n-1}=a)
&=
\sum_{\ell=0}^{n-2}\P_0(Y_\ell=0)\,
\P_0(Y_{n-\ell-1}=a,\sigma_0^+>n-\ell-1).
\label{eq:last-zero-decomp-correct}
\end{align}
Since $\P_0(Y_\ell=0)=r_{\ell+1}$, summing
\eqref{eq:last-zero-decomp-correct} over $1\le a\le j-1$, we have
\begin{align}
r_n^{(j)}-r_n
&=
\sum_{a=1}^{j-1}\P_0(Y_{n-1}=a) \notag\\
&=
\sum_{\ell=0}^{n-2}r_{\ell+1}
\P_0(Y_{n-\ell-1}\in A_j,\sigma_0^+>n-\ell-1) \notag\\
&=
\sum_{m=1}^{n-1} r_{n-m}\,
\P_0(Y_m\in A_j,\sigma_0^+>m).
\label{eq:sum-over-m}
\end{align}

For $m\ge1$, let
$$
a_m:=\P_0(Y_m\in A_j,\sigma_0^+>m).
$$
We split the sum in \eqref{eq:sum-over-m} into the ranges $m<j^2$ and
$m\ge j^2$.

\smallskip
\noindent
\textbf{Range 1: $m<j^2$.} 
By \eqref{eq:sigma0-plus-upper-direct},
$
a_m\le \P_0(\sigma_0^+>m)\le C_1(m+1)^{-1/2}.
$
Using also \eqref{eq:r_n-upper} and \eqref{eq:j2-small}, we have
\begin{align*}
I_1
&:=
\sum_{m=1}^{j^2-1} r_{n-m}a_m
\le
C_2\sum_{m=1}^{j^2-1}(n-m)^{-1/2}(m+1)^{-1/2} \le
C_3 n^{-1/2}\sum_{m=1}^{j^2-1}(m+1)^{-1/2}
\le
C_4\frac{j}{\sqrt n}.
\end{align*}

\smallskip
\noindent
\textbf{Range 2: $j^2\le m\le n-1$.}
By \eqref{eq:r_n-upper},
$$
I_2:=
\sum_{m=j^2}^{n-1} r_{n-m}a_m \le
C_0\sum_{m=j^2}^{n-1}\frac{a_m}{\sqrt{n-m}}.
$$
We split this sum further. First consider $j^2\le m\le n/2$. Since $(n-m)^{-1/2}\le C_5n^{-1/2}$ on this
range, by Lemma~\ref{lem:interval-killed}, we have
\begin{align}\label{I2.bound1}
\sum_{j^2\le m\le n/2}\frac{a_m}{\sqrt{n-m}}
&\le
C_5 n^{-1/2}
\sum_{j^2\le m\le n/2}a_m \le
C_6 n^{-1/2}
\sum_{\substack{k\ge0\, :\, 2^k j^2\le n/2}}
j^2(2^k j^2+1)^{-1/2}
\le
C_7\frac{j}{\sqrt n}.
\end{align}
It remains to consider $n/2<m\le n-1$. Note that 
$$
\sum_{j^2\le m\le n/2}\frac{a_m}{\sqrt{n-m}}=\sum_{1\le k<n/2}\frac{a_{n-k}}{\sqrt k}.
$$
We split the latter sum into $k<j^2$ and $k\ge j^2$. If $k<j^2$, then by
\eqref{eq:sigma0-plus-upper-direct} and \eqref{eq:j2-small}, we notice that
$$
a_{n-k}\le C_1(n-k+1)^{-1/2}\le C_8 n^{-1/2}.
$$
Therefore
$$
\sum_{1\le k<j^2}\frac{a_{n-k}}{\sqrt k}
\le
C_8 n^{-1/2}\sum_{1\le k<j^2}k^{-1/2}
\le
C_9\frac{j}{\sqrt n}.
$$
Finally, consider $j^2\le k<n/2$. Notice that
\begin{align*}
\sum_{j^2\le k<n/2}\frac{a_{n-k}}{\sqrt k}
&\le
\sum_{\substack{l\ge0\, :\, 2^l j^2<n/2}}
(2^l j^2)^{-1/2}
\sum_{2^l j^2\le k<2^{l+1}j^2}a_{n-k}.
\end{align*}
For each $l$ in the last sum, the indices $n-k$ lie in the interval
$[\lfloor n/2\rfloor,n-1]$. Hence, using Lemma~\ref{lem:interval-killed} with
$M=\lfloor n/2\rfloor$, we have
$$
\sum_{2^l j^2\le k<2^{l+1}j^2}a_{n-k}
\le
\sum_{m=\lfloor n/2\rfloor}^{n-1}a_m
\le
C_{10} j^2(n+1)^{-1/2}.
$$
Consequently,
\begin{align}\label{I2.bound2}
\sum_{j^2\le m\le n/2}\frac{a_m}{\sqrt{n-m}}=\sum_{j^2\le k<n/2}\frac{a_{n-k}}{\sqrt k}
&\le
C_{10} j^2(n+1)^{-1/2}
\sum_{\substack{l\ge0\, :\, 2^l j^2<n/2}}
(2^l j^2)^{-1/2}
\le
C_{11}\frac{j}{\sqrt n}.
\end{align}
Combining \eqref{I2.bound1} and \eqref{I2.bound2}, we get
$$
I_2\le C_{12}\frac{j}{\sqrt n}.
$$
Combining \eqref{eq:sum-over-m}, the bounds on $I_1$ and $I_2$, and
\eqref{eq:r_n-upper}, we obtain
$$
r_n^{(j)}
\le
r_n+I_1+I_2
\le
C_0 n^{-1/2}+C_{13}\frac{j}{\sqrt n}
\le
C_{14}\frac{j}{\sqrt n}.
$$
This completes the proof.
\end{proof}

\subsection{Markovian structure of forward steps}\label{sec:forward-chain}

For each $x\in\{1,\ldots,n\}$, define
$$
F(x,n):=\sum_{k=0}^{\tau_0^+-1}\mathbf{1}_{\{\widetilde X_k=x-1,\ \widetilde X_{k+1}=x\}},
$$
the number of forward crossings from $x-1$ to $x$ before the first return of
$\widetilde \X$ to $0$. Clearly,
$
F(1,n)=1$
 a.s.
and
$$
\{F(n,n)\ge 1\}=\{\tau_n<\tau_0^+\}.
$$
Hence
\begin{equation}\label{eq:forward-ruin}
\P(F(n,n)\ge 1)=r_n.
\end{equation}

Similarly as Lemma \ref{prop:auxiliary-transition} for the backward local-time sequence $(B(n-x,n))_{0\le x\le n-1}$, the next lemma show that the forward local-time sequence $(F(x,n))_{1\le x\le n}$ is also Markovian.
\begin{lemma}\label{lem:forward-chain}
For every $n\ge1$,
$$
(F(1,n),F(2,n),\dots,F(n,n))
\stackrel d=
(Z_1,Z_2,\dots,Z_n),
$$
where $(Z_k)_{k\ge1}$ is a Markov chain on $\mathbb Z_+$ with $Z_1=1$, which is absorbed at $0$, and its transition probabilities given by
$$
\P(Z_{k+1}=y\mid Z_k=z)=P^{\,z}(-1,y-z-1),
\quad z\ge1,\, y\ge0.
$$
In particular,
$$
\P(Z_n\ge 1)=r_n.
$$
\end{lemma}

\begin{proof}
It is sufficient to show that for every $x\in\{1,\dots,n-1\}$ and every $y,z\ge0$,
\begin{equation}\label{eq:forward-kernel}
\P(F(x+1,n)=y\mid F(x,n)=z)=P^{\,z}(-1,y-z-1).
\end{equation}

Fix $x\in\{1,\dots,n-1\}$ and condition on the event $\{F(x,n)=z\}$.
Then the edge $\{x-1,x\}$ is crossed exactly $z$ times from left to right before
$\tau_0^+$. These $z$ forward crossings split the trajectory into $z$ successive
stages. During the $j$-th stage, the walk starts at $x$ immediately after the
$j$-th jump from $x-1$ to $x$, makes some number of jumps from $x$ to $x+1$, and
eventually leaves the stage by a jump from $x$ to $x-1$. For $1\le j\le z$, let $\ell_j(x,n)$ be the number of jumps from $x$ to $x+1$
during the $j$-th stage. Then
$$
F(x+1,n)=\sum_{j=1}^{z}\ell_j(x,n).
$$
Fix integers $s_1,\dots,s_z\ge 0$ such that
$
s_1+\cdots+s_z=y,
$
and define
$$
u_0:=-1,
\quad
u_j:=\sum_{i=1}^j s_i-j-1,
\quad 1\le j\le z.
$$
Thus $u_j=u_{j-1}+s_j-1$. At the beginning of the $j$-th stage, the edge $\{x-1,x\}$ has been crossed $2j-1$ times, while the edge $\{x,x+1\}$ has been crossed $2\sum_{i=1}^{j-1}s_i$ times. Thus the probability of exactly $s_j$ jumps from $x$ to $x+1$ during this stage and then one jump from $x$ to $x-1$ is $P(u_{j-1},u_j)$. Hence,
\begin{align*}
&\mathbb P\bigl(\ell_j(x,n)=s_j \,\big|\, \ell_1(x,n)=s_1,\dots,\ell_{j-1}(x,n)=s_{j-1},\,F(x,n)=z\bigr)
=
P(u_{j-1},u_j).
\end{align*}
Therefore, 
\begin{align*}
&\P\bigl(\ell_1(x,n)=s_1,\dots,\ell_z(x,n)=s_z\mid F(x,n)=z\bigr)=
P(u_0,u_1)\,P(u_1,u_2)\cdots P(u_{z-1},u_z).
\end{align*}
Similarly as in the proof of Lemma~\ref{prop:auxiliary-transition}, summing over all sequences $(s_1,\dots,s_z)$ such that
$s_1+\cdots+s_z=y$, we obtain
$$
\P(F(x+1,n)=y\mid F(x,n)=z)=P^{\,z}(-1,y-z-1),
$$
which is exactly \eqref{eq:forward-kernel}. The identity $\P(Z_n\ge1)=r_n$ follows from \eqref{eq:forward-ruin}.
This proves the lemma.
\end{proof}

We next record moment bounds for the forward local-time chain $(Z_n)_{n\ge 1}$.
\begin{lemma}\label{lem:forward-moments}
There exists a constant $C<\infty$ such that for all $n\ge1$,
$$
\E[Z_n]\le C
\quad\text{and}\quad
\E[Z_n^2]\le C\sqrt n.
$$
\end{lemma}

\begin{proof}
Let
$
\upsilon_0:=\inf\{k\ge 1:Z_k=0\}
$
be the absorbing time of $Z$. 
Then for $n\ge 1$, we have
$\{\upsilon_0>n\}=\{Z_n\ge 1\}$. 
By Lemma~\ref{lem:forward-chain},
$
\P(\upsilon_0>n)=\P(Z_n\ge 1)=r_n.
$
Hence, by Proposition~\ref{thm:ruin-order},
\begin{equation}\label{eq:forward-survival}
\P(\upsilon_0>n)\le C_0\,n^{-1/2}
\quad\text{for } n\ge 1.
\end{equation}

We first prove the second-moment bound. For $z\ge 1$, using the transition probabilities
of $Z$, we have
\begin{align*}
\E[(Z_{k+1}-Z_k)^q\mid Z_k=z]
&=
\sum_{j\in\mathbb Z} j^q\,P^{\,z}(-1,j-1) \quad\text{for } q\in\{1,2\}.
\end{align*}
By the same argument as in Lemma~\ref{lem:step_distr}, there exist constants $C_1<\infty$ and $c_1>0$ such that
\begin{equation}\label{eq:forward-step-estimates}
\left|\E[Z_{k+1}-Z_k\mid Z_k=z]\right|\le C_1e^{-c_1 z},
\quad
\left|\E[(Z_{k+1}-Z_k)^2\mid Z_k=z]-\varsigma_\beta^2\right|\le C_1e^{-c_1 z}
\quad \text{for }z\ge 1.
\end{equation}
Therefore,
\begin{align*}
\E[Z_{k+1}^2-Z_k^2\mid Z_k=z]
&=
2z\,\E[Z_{k+1}-Z_k\mid Z_k=z]
+\E[(Z_{k+1}-Z_k)^2\mid Z_k=z]
\end{align*}
is uniformly bounded above in $z\ge 1$. Hence there exists $C_2<\infty$ such that
$$
\E[Z_{k+1}^2\mid Z_k=z]\le z^2+C_2
\quad\text{for all }z\ge 1.
$$
Since $0$ is absorbing, the same inequality also holds for $z=0$. Therefore
$Z_{k\wedge\upsilon_0}^2-C_2(k\wedge\upsilon_0)$ with $k\ge 1, $
is a supermartingale. Thus, for every $n\ge 1$,
$
\E[Z_{n\wedge\upsilon_0}^2-C_2(n\wedge\upsilon_0)]
\le
\E[Z_1^2-C_2].
$
Since $Z_{n\wedge\upsilon_0}=Z_n$ and $Z_1=1$, this yields
$$
\E[Z_n^2]\le 1-C_2+C_2\,\E[n\wedge\upsilon_0].
$$
Using \eqref{eq:forward-survival}, we obtain
$$
\E[n\wedge\upsilon_0]
=
\sum_{k=0}^{n-1}\P(\upsilon_0>k)
\le
1+\sum_{k=1}^{n-1}C_0\,k^{-1/2}
\le
C_3\sqrt n.
$$
Therefore,
$
\E[Z_n^2]\le C_4\sqrt n.
$
Finally, using the Cauchy--Schwarz inequqality and \eqref{eq:forward-survival}, we obtain
\begin{align*}
\E[Z_n]
&=
\E[Z_n;\upsilon_0>n]\le
\bigl(\E[Z_n^2]\bigr)^{1/2}\,\P(\upsilon_0>n)^{1/2}\le
(C_4\sqrt n)^{1/2}(C_0n^{-1/2})^{1/2}
\le C_5.
\end{align*}
This completes the proof.
\end{proof}

\section{Proof of Theorem \ref{thm:main}}
\label{sec:proof-main}

\subsection{Quasi-independent percolation}\label{sec:percolation}
Let $\Tcal=(V,E)$ be an infinite locally finite tree. For each edge $e\in E$, we assign a Bernoulli random variable $\xi_e$ with parameter $p_e\in [0,1]$. The Bernoulli field $(\xi_e)_{e\in E}$ is not necessarily independent nor identically distributed. We say that $e$ is \textbf{open} if $\xi_e=1$, and \textbf{closed} otherwise. Assume that $(\xi_e)_{e\in E}$ is governed by a probability measure $\mathbf{Q}$. We call $\mathbf Q$ a \textbf{bond percolation} on $\Tcal$. 
After removing all closed edges from $E$, we obtain connected components consisting of open edges, which we call \textbf{clusters}. 

If two vertices $x$ and $y$ belong to the same cluster, we write $x\leftrightarrow y$. If the cluster containing $x$ has infinitely many vertices, we write $x\leftrightarrow \infty$. For two vertices $x,y$, we denote by $x\wedge y$ their nearest common ancestor.
\begin{definition}\label{def:quasi}
    A bond percolation $\mathbf{Q}$ is \textbf{quasi-independent} if there exists a  constant $M\in (0,\infty)$ such that \begin{align}
\label{def:weak-QI}
   & \mathbf{Q}(\rho\leftrightarrow x, \rho\leftrightarrow y \mid \rho\leftrightarrow x\wedge y )\le M\cdot \mathbf{Q}(\rho\leftrightarrow x \mid \rho\leftrightarrow x\wedge y )\mathbf{Q}(\rho\leftrightarrow y \mid \rho\leftrightarrow x\wedge y )
\end{align}
for each $x, y\in V$. 
\end{definition}

For each edge $e\in E$, we denote its two endpoints by $e^-$ and $e^+$ where $|e^+|=|e^-|+1$. Let 
\begin{align}\label{adt.c} \text{ $c(e) := 1$ for $|e| = 1$}\quad \text{and}\quad
    c(e) :=\frac{\mathbf{Q}(\rho\leftrightarrow e^+)}{\mathbf{Q}(e \text{ is closed} \mid \rho\leftrightarrow e^-)}.
\end{align}
for $|e|>1$. We call $(c(e))_{e\in E}$ the \textbf{adapted conductances} of the percolation $\mathbf{Q}$. We will use the following result.
\begin{proposition}[Theorem 5.19 in \cite{LP2016}]\label{lyons1989}
    Let $\mathbf{Q}$ be a quasi-independent percolation process taking place on an infinite locally finite tree $\Tcal = (V,E)$.
    \begin{itemize}
        \item [(i)] If \(\inf_{\pi \in \Pi} \sum_{e \in \pi} \mathbf{Q}(\rho \leftrightarrow e^+) = 0 \text{ then } \mathbf{Q}(\rho \leftrightarrow \infty) = 0\);
        \item [(ii)] If there exists a non-zero flow $\theta$ such that $\sum_{e\in E} \frac{\theta(e)^2}{c(e)}<\infty$ then  $\mathbf{Q}(\rho \leftrightarrow \infty) > 0.$
    \end{itemize}
\end{proposition}

\subsection{Ruin percolation}\label{sec:ruin-per}
Let $\tau_v=\inf\{n\ge0: X_n=v\}$ and $\tau_v^+=\inf\{n>\tau_v: X_n=v\}$ be respectively the first hitting time of vertex $v$ and the first return time to vertex $v$. 
Define
$$
\mathcal{C} ( \rho)=\left\{\{v^{-1},v\} \in E \colon \tau_v  < \tau_{\rho}^+ \right\}. 
$$ 

Let $\tau_u(v)$ and $\tau_{u}^{+}(v)$ be respectively the first hitting times and the return time to vertex $u$ associated with $\X^{(v)}$.
Let 
$$\mathcal{C}_{\rm CP} ( \rho)=\left\{\{v^{-1},v\} \in E \colon \tau_v(v)  < \tau_{\rho}^+(v) \right\}. $$
We say an edge $e \in E$ is \textbf{open} if $e \in{\mathcal{C}}_{\mathrm{CP}} ( \rho)$, and \textbf{closed} otherwise. We define a correlated percolation by removing all closed edges, and we refer to this model as the \textbf{ruin percolation}.

 We will use the following result:
\begin{lemma}[Lemma 3.3 in \cite{LN2026}, Lemma 7.1 in \cite{CKS2020}] \label{per2.lem} We have
\begin{align*}
    {\P}(\tau^+_{\rho} = \infty) = {\P} (|\mathcal{C}(\rho)| = \infty) = {\P}(|\mathcal{C}_{\rm CP}(\rho)| = \infty).
\end{align*} 
Consequently, the process $\X$ \begin{itemize}
    \item is a.s. recurrent if $ \P(|\mathcal{C}_{\rm CP}(\rho)| < \infty)=1$, 
\item is a.s. transient if $\P(|\mathcal{C}_{\rm CP}(\rho)| = \infty)>0$.
\end{itemize}
\end{lemma}

In this section we aim to prove that the ruin percolation is quasi-independent.
 
 \begin{proposition}
\label{prop:per}
There exists a constant $M<\infty$ such that for every pair of edges
$e_1,e_2\in E$ with $e=e_1\wedge e_2$ being the last common edge of $\mathcal{P}_{e_1}$ and $\mathcal{P}_{e_2}$, 
\begin{equation}\label{eq:edge-qi-final}
\P\bigl(e_1,e_2\in\mathcal C_{\rm CP}(\rho)\mid e\in\mathcal C_{\rm CP}(\rho)\bigr)
\le
M\,
\P\bigl(e_1\in\mathcal C_{\rm CP}(\rho)\mid e\in\mathcal C_{\rm CP}(\rho)\bigr)
\P\bigl(e_2\in\mathcal C_{\rm CP}(\rho)\mid e\in\mathcal C_{\rm CP}(\rho)\bigr).
\end{equation}
Equivalently, the ruin percolation is quasi-independent.
\end{proposition}

Fix $e_1, e_2\in E$ and $e=e_1\wedge e_2$.
Let $\widehat{\X}=(\widehat X_k)_{k\ge 0}$ be the extension process of $\X$ on the subtree $\mathcal{P}_{e_1}\cup \mathcal{P}_{e_2}$. We construct $\widehat{\X}$ using the same exponential random variables similarly as in Section \ref{sec:strong.constr}. 
Let
$$
\widehat\tau_\rho^+:=\inf\{k\ge1:\widehat X_k=\rho\}.
$$
be the first return time to $\rho$. 
For $e=\{e^{-},e^{+}\}\in E$ with $|e^{+}|=|e^{-}|+1$, let 
$$
N_e:=
\sum_{k=0}^{\widehat \tau_\rho^+-1}\mathbf{1}_{\{\widehat X_k=e^{-},\, \widehat X_{k+1}=e^{+}\}}.
$$
be the number of down-crossings from $e^-$ to $e^+$ by the first return time to $\rho$. 

\begin{lemma}
\label{lem:percolation-localtime}
For every $e\in E$,
$$
\bigl\{e\in \mathcal C_{\rm CP}(\rho)\bigr\}
=
\{N_e\ge 1\}.
$$
\end{lemma}
\begin{proof}
By definition,
$e=\{e^-, e^+\}\in \mathcal C_{\rm CP}(\rho)$ if and only if
$\tau_{e^+}{(e^+)}<\tau_\rho^+{(e^+)}$, that is, the extension process $\X^{(e^+)}$ hits
$e^+$ before returning to $\rho$. As the subtree $\mathcal{P}_{e_1}\cup \mathcal{P}_{e_2}$ is finite, the extension process $\widehat \X$ visits $\mathcal{P}_e$ infinitely many times. By the restriction principle in
Lemma~\ref{lem:rest}, the extension process $\X^{(e^+)}$ coincides with the restriction of $\widehat \X$ to
$\mathcal P_e$. Hence the event $e\in \mathcal C_{\rm CP}(\rho)$ is exactly the event
that the first excursion of $\widehat \X$ from $\rho$ crosses the edge
$\{e^-,e^+\}$ downward at least once, that is, $N_e\ge 1$.
\end{proof}

\begin{proposition}
\label{prop:mu-upper}
For $j\ge 1$, let
$$
\vartheta_e(j):=
\P\bigl(N_e=j\mid N_e\ge 1\bigr).
$$
Then 
there exists a constant $C\in (0,\infty)$  such that
$$\sum_{j\ge1} j\,\vartheta_e(j)\le C\sqrt{|e|}
\quad\text{and}\quad
\sum_{j\ge1} j^2\,\vartheta_e(j)\le C|e|.$$
\end{proposition}

\begin{proof}
Let $n:=|e|$. We first estimate
$
\P(N_e\ge 1).
$
Consider the restriction of $\widehat{\X}$ to the path $\mathcal P_e$. By the
restriction principle, this restriction process has the same law as the
one-dimensional TSAW on $\{0,1,\dots,n\}$ up to its first return to $0$, after
identifying $\rho$ with $0$ and $e^+$ with $n$. In particular, the crossings from $e^-$ to $e^+$ by $\widehat{\X}$ before the first return
to $\rho$ are exactly the forward crossings of the last edge $\{n-1,n\}$ by the
restriction process before the first return to $0$. Therefore
$$
\P(N_e\ge 1)=r_n.
$$
By Proposition~\ref{thm:ruin-order}, there exists a constant $C_1>0$ such that
\begin{equation}\label{eq:mu-denom}
\P(N_e\ge 1)\ge C_1\,n^{-1/2}.
\end{equation}

We now estimate the moments of $N_e$. Let
$$
F(n,n):=
\sum_{k=0}^{\tau_0^+-1}
\mathbf 1_{\{\widetilde X_k=n-1,\ \widetilde X_{k+1}=n\}}
$$
be the number of forward crossings of the last edge by the one-dimensional TSAW
on $\{0,1,\dots,n\}$ before the first return to $0$. By the restriction
principle, we note that
\begin{equation}\label{eq:Ne-F-law}
N_e\stackrel d=F(n,n).
\end{equation}
By Lemma~\ref{lem:forward-chain}, $F(n,n)$ has the same distribution as $Z_n$.
Consequently, for every integer $q\ge 1$,
\begin{equation}\label{eq:mu-num-moment}
\sum_{j\ge 1} j^q\,\P(N_e=j)
=
\E[Z_n^q].
\end{equation}
By Lemma~\ref{lem:forward-moments}, there exists a constant $C_2<\infty$ such
that
\begin{equation}\label{eq:fY-moment-used}
\E[ Z_n]\le C_2,\quad
\E[ Z_n^2]\le C_2\sqrt n.
\end{equation}
Combining \eqref{eq:mu-denom}, \eqref{eq:mu-num-moment}, and
\eqref{eq:fY-moment-used}, we obtain
\begin{align*}
\sum_{j\ge1} j\,\vartheta_e(j)
&=
\frac{\sum_{j\ge1} j\,\P(N_e=j)}{\P(N_e\ge1)}
\le
C\,\sqrt n,
\\
\sum_{j\ge1} j^2\,\vartheta_e(j)
&=
\frac{\sum_{j\ge1} j^2\,\P(N_e=j)}{\P(N_e\ge1)}
\le
C\,n.
\end{align*}
This completes the proof.
\end{proof}

\begin{lemma}\label{lem:qj-upper}
Let $e_1,e_2\in E$ with $e=e_1\wedge e_2$, and assume that $e_1$ and $e_2$
lie strictly below $e$ in two distinct descendant subtrees. 
Then, there exists a constant $C<\infty$ such that for all $j\ge1$,
$$
\P\bigl(e_1,e_2\in\mathcal C_{\rm CP}(\rho)\mid N_e=j\bigr)
\le
C
\left(1\wedge \frac{j}{\sqrt{|e_1|-|e|+1}}\right)
\left(1\wedge \frac{j}{\sqrt{|e_2|-|e|+1}}\right).
$$
\end{lemma}

\begin{proof}
For $i\in \{1,2\}$, set
$
m_i:=|e_i|-|e|+1
$ and write the unique path from $e^+$ to
$e_i^+$ as 
$(u_{i,1},u_{i,2},\dots,u_{i,m_i})$ with $u_{i,1}=e^+$ and $u_{i,m_i}=e_i^+$. Let $L_i$ be the number of crossings from $e^+$ to $u_{i,2}$ by $\widehat{\X}$ before time $\widehat\tau_\rho^+$. 

Fix $j\ge 1$. We now work on the event $\{N_e=j\}$, i.e. the number of
crossings from $e^+$ to $e^-$ by $\widehat{\X}$ before $\widehat\tau_\rho^+$
is exactly $j$. We first estimate the conditional moments of $L_i$ on this
event. By the restriction principle, the restriction of $\widehat{\X}$ to
$\mathcal P_e$ coincides with the extension process $\X^{(e^+)}$ up to the
first return to $\rho$. In particular, $N_e$ is also equal to the number of
crossings by $\X^{(e^+)}$ from $e^+$ to $e^-$ before its first return to
$\rho$. Since $e^+$ is the endpoint of the path $\mathcal P_e$, right after
each visit to $e^+$, the extension process $\X^{(e^+)}$ jumps from $e^+$ to
$e^-$ deterministically. Therefore, by the strong construction of
$\X^{(e^+)}$, $N_e$ is independent of the exponential variables
$
\bigl\{\xi(e^+,y,\ell): y\sim e^+,\ \ell\ge0\bigr\}
$
and of the exponential variables with first coordinate in
$
\{u_{1,2},\dots,u_{1,m_1}\}\cup\{u_{2,2},\dots,u_{2,m_2}\}.
$
Fix $i\in\{1,2\}$. For $\ell\ge0$, define
$$
T_0(\ell):=
\sum_{q=0}^{\ell}
\frac{\xi(e^+,e^-,q)}{w(2q+1)}\quad\text{and}\quad
T_i(\ell):=
\sum_{q=0}^{\ell}
\frac{\xi(e^+,u_{i,2},q)}{w(2q)}.
$$
Let
$$
\overline L_i:=
\sum_{\ell\ge0}\mathbf 1_{\{T_i(\ell)<T_0(j-1)\}},
$$
which is equal to the number of crossings from $e^+$ to
$u_{i,2}$ before the $j$-th crossing from $e^+$ to $e^-$ when only the two
oriented edges $(e^+,e^-)$ and $(e^+,u_{i,2})$ are kept, with the same
exponential clocks on these two oriented edges. Since the crossing order
from $e^+$ is obtained by ordering the clock times corresponding to all
oriented edges with tail $e^+$, the number of crossings from $e^+$ to
$u_{i,2}$ before the $j$-th crossing by $\widehat{\X}$ from $e^+$ to $e^-$ is
at most $\overline L_i$. Hence, on $\{N_e=j\}$, we have
$$
L_i\le \overline L_i.
$$
Moreover, $\overline L_i$ is independent of $\{N_e=j\}$.

By the same argument as in Lemma~\ref{lem:forward-chain}, we have
$$
\P(\overline L_i=y)=P^j(-1,y-j-1)
\quad\text{for each } y\ge0.
$$
Equivalently, if $(\eta_n)$ is the Markov chain with transition probability matrix $P$ which starts from $-1$, then
$\overline L_i$ has the same distribution as $j+1+\eta_j$. By
Lemma~\ref{lem:expconv} and the remark following it, applied with initial
state $-1$, there exists a constant $C_1<\infty$ such that, for all $j\ge1$,
$$
\E[\overline L_i]\le C_1 j\quad\text{and} \quad\E[\overline L_i^2]\le C_1 j^2.
$$
Consequently,
\begin{equation}\label{eq:Li-moments}
\E[L_i\mid N_e=j]\le C_1j,\quad
\E[L_i^2\mid N_e=j]\le C_1j^2\quad\text{for } i=1,2.
\end{equation}

By definition of $\mathcal C_{\rm CP}(\rho)$ and by the restriction principle,
for $i\in\{1,2\}$, the event $\{e_i\in\mathcal C_{\rm CP}(\rho)\}$ is the event
that the extension process $\X^{(e_i^+)}$ reaches $e_i^+$ before its first
return to $\rho$, which is the same as the event that the restriction of
$\widehat{\X}$ to $\mathcal P_{e_i^+}$ reaches $e_i^+$ before
$\widehat\tau_\rho^+$.

We now condition on the event $\{N_e=j\}$ and on the exponential variables
$\{\xi(e^+,y,\ell):y\sim e^+,\ell\ge0\}$. Under this conditioning, the values
of $L_1$ and $L_2$ are determined. Moreover, the exponential variables with
first coordinate in
$\{u_{1,2},\dots,u_{1,m_1}\}$ and the exponential variables with first
coordinate in
$\{u_{2,2},\dots,u_{2,m_2}\}$ are independent. For fixed $i\in\{1,2\}$, under the conditioning above, the restriction of
$\widehat{\X}$ to the path $(e^+,u_{i,2},\dots,u_{i,m_i})$ has exactly $L_i$
crossings from $e^+$ to $u_{i,2}$ before time $\widehat\tau_\rho^+$. Hence the
event that the restriction of $\widehat{\X}$ to $\mathcal P_{e_i^+}$ reaches
$e_i^+$ before $\widehat\tau_\rho^+$ is the same as the event that the
restriction of $\widehat{\X}$ to $(e^+,u_{i,2},\dots,u_{i,m_i})$ reaches
$e_i^+$ before its $L_i$-th return to $e^+$. By the restriction principle,
after identifying $e^+$ with $0$ and $e_i^+$ with $m_i-1$, the conditional
probability of this event is equal to $r_{m_i-1}^{(L_i)}$, where we use the
convention that $r_m^{(0)}:=0$.
Therefore
\begin{align}
&\P\bigl(e_1,e_2\in\mathcal C_{\rm CP}(\rho)\mid N_e=j\bigr) =
\E\left[
r_{m_1-1}^{(L_1)}r_{m_2-1}^{(L_2)}
\,\middle|\, N_e=j
\right].
\label{eq:qj-ruin-bound}
\end{align}
By Proposition~\ref{prop:gen-ruin-upper}, we have, for all $\ell\ge0$,
$$
r_{m_i-1}^{(\ell)}
\le
C_2\left(1\wedge \frac{\ell}{\sqrt{m_i}}\right),
\quad i\in \{1,2\}.
$$
It thus follows from \eqref{eq:qj-ruin-bound} that
\begin{align}
&\P\bigl(e_1,e_2\in\mathcal C_{\rm CP}(\rho)\mid N_e=j\bigr) \le
C_2\,
\E\left[
\left(1\wedge \frac{L_1}{\sqrt{m_1}}\right)
\left(1\wedge \frac{L_2}{\sqrt{m_2}}\right)
\,\middle|\, N_e=j
\right].
\label{eq:qj-L-bound}
\end{align}

It remains to estimate the last expectation. Without loss of generality assume
$m_1\le m_2$.

If $m_1\le j^2$ and $m_2\le j^2$, then the expectation in
\eqref{eq:qj-L-bound} is at most $1$, while
$$
\left(1\wedge \frac{j}{\sqrt{m_1}}\right)
\left(1\wedge \frac{j}{\sqrt{m_2}}\right)=1.
$$

If $m_1\le j^2<m_2$, then
$$
\left(1\wedge \frac{L_1}{\sqrt{m_1}}\right)
\left(1\wedge \frac{L_2}{\sqrt{m_2}}\right)
\le
\frac{L_2}{\sqrt{m_2}}.
$$
Using \eqref{eq:Li-moments}, we get
$$
\E\left[
\left(1\wedge \frac{L_1}{\sqrt{m_1}}\right)
\left(1\wedge \frac{L_2}{\sqrt{m_2}}\right)
\,\middle|\, N_e=j
\right]
\le
C_1\frac{j}{\sqrt{m_2}}= C_1\left(1\wedge \frac{j}{\sqrt{m_1}}\right)
\left(1\wedge \frac{j}{\sqrt{m_2}}\right).
$$

Finally, suppose that $j^2<m_1\le m_2$. Then
$$
\left(1\wedge \frac{L_1}{\sqrt{m_1}}\right)
\left(1\wedge \frac{L_2}{\sqrt{m_2}}\right)
\le
\frac{L_1L_2}{\sqrt{m_1m_2}}.
$$
By H\"older's inequality and \eqref{eq:Li-moments},
$$
\E[L_1L_2\mid N_e=j]
\le
\Bigl(\E[L_1^2\mid N_e=j]\E[L_2^2\mid N_e=j]\Bigr)^{1/2}
\le
C_1 j^2.
$$
Thus
$$
\E\left[
\left(1\wedge \frac{L_1}{\sqrt{m_1}}\right)
\left(1\wedge \frac{L_2}{\sqrt{m_2}}\right)
\,\middle|\, N_e=j
\right]
\le
C_1\frac{j^2}{\sqrt{m_1m_2}}=C_1\left(1\wedge \frac{j}{\sqrt{m_1}}\right)
\left(1\wedge \frac{j}{\sqrt{m_2}}\right)
$$

Combining the three cases with \eqref{eq:qj-L-bound}, we obtain the desired
estimate.
\end{proof}

\begin{proof}[Proof of Proposition \ref{prop:per}]
Fix two edges $e_1,e_2\in E$, and let
$e:=e_1\wedge e_2$ be the last common edge of $\mathcal P_{e_1}$ and $\mathcal P_{e_2}$. If
$e=e_1$ or $e=e_2$, then \eqref{eq:edge-qi-final} is immediate, since for
example when $e=e_1$,
$$
\P\bigl(e_1,e_2\in \mathcal C_{\rm CP}(\rho)\mid e\in \mathcal C_{\rm CP}(\rho)\bigr)
=
\P\bigl(e_2\in \mathcal C_{\rm CP}(\rho)\mid e\in \mathcal C_{\rm CP}(\rho)\bigr),
$$
while
$
\P\bigl(e_1\in \mathcal C_{\rm CP}(\rho)\mid e\in \mathcal C_{\rm CP}(\rho)\bigr)=1.
$
Thus it remains to consider the case where both $e_1$ and $e_2$ lie strictly
below $e$ in two distinct descendant subtrees of $e^+$. Set
$$
n:=|e|,
\quad
m_i:=|e_i|-|e|+1\in\N
\quad\text{for } i\in \{1,2\}.
$$
By Lemma~\ref{lem:percolation-localtime},
$
\{e\in\mathcal C_{\rm CP}(\rho)\}=\{N_e\ge1\}.
$
Therefore, using the law of total probability,
\begin{align}
&\P\bigl(e_1,e_2\in \mathcal C_{\rm CP}(\rho)\mid e\in \mathcal C_{\rm CP}(\rho)\bigr)  =
\sum_{j\ge1}\vartheta_e(j)\,
\P\bigl(e_1,e_2\in \mathcal C_{\rm CP}(\rho)\mid N_e=j\bigr).
\label{eq:branch-mixture}
\end{align}

By Lemma~\ref{lem:qj-upper}, there exists a constant $C_1<\infty$ such that
\begin{equation}\label{eq:qj-upper-per}
\P\bigl(e_1,e_2\in \mathcal C_{\rm CP}(\rho)\mid N_e=j\bigr)
\le
C_1
\Bigl(1\wedge \frac{j}{\sqrt{m_1}}\Bigr)
\Bigl(1\wedge \frac{j}{\sqrt{m_2}}\Bigr)
\quad\text{for all } j\ge1.
\end{equation}
Also, by Proposition~\ref{prop:mu-upper}, there exists a constant $C_2<\infty$ such that
\begin{equation}\label{eq:mu-upper-per}
\sum_{j\ge1} j\,\vartheta_e(j)\le C_2\sqrt{n}
\quad\text{and}\quad
\sum_{j\ge1} j^2\,\vartheta_e(j)\le C_2n.
\end{equation}
Substituting \eqref{eq:qj-upper-per} into \eqref{eq:branch-mixture}, we obtain
\begin{align}
&\P\bigl(e_1,e_2\in \mathcal C_{\rm CP}(\rho)\mid e\in \mathcal C_{\rm CP}(\rho)\bigr)\le
C_1
\sum_{j\ge 1}
\vartheta_e(j)
\Bigl(1\wedge \frac{j}{\sqrt{m_1}}\Bigr)
\Bigl(1\wedge \frac{j}{\sqrt{m_2}}\Bigr).
\label{eq:sum-to-bound}
\end{align}

We now estimate the sum on the right-hand side. Without loss of generality,
assume that $m_1\le m_2$.

\smallskip
\textit{Case 1: $m_1\le n$ and $m_2\le n$.}
Since
$
\Bigl(1\wedge \frac{j}{\sqrt{m_1}}\Bigr)
\Bigl(1\wedge \frac{j}{\sqrt{m_2}}\Bigr)\le 1
$, the sum is bounded by
$
C_1\sum_{j\ge 1}
\vartheta_e(j)=C_1.
$
Therefore,
\begin{equation}\label{eq:sum-case1}
\P(e_1,e_2\in\mathcal C_{\rm CP}(\rho)\mid e\in\mathcal C_{\rm CP}(\rho))
\le C_1
\prod_{i=1}^2\Bigl(1\wedge \sqrt{\frac{n}{m_i}}\Bigr).
\end{equation}

\smallskip
\textit{Case 2: $m_1\le n< m_2$.}
Using the fact that
$
\Bigl(1\wedge \frac{j}{\sqrt{m_1}}\Bigr)
\Bigl(1\wedge \frac{j}{\sqrt{m_2}}\Bigr)\le \frac{j}{\sqrt{m_2}}
$
and \eqref{eq:mu-upper-per}, we obtain
\begin{equation}\label{eq:sum-case2}
\P(e_1,e_2\in\mathcal C_{\rm CP}(\rho)\mid e\in\mathcal C_{\rm CP}(\rho))
\le
\frac{C_1}{\sqrt{m_2}}
\sum_{j\ge1} j\,\vartheta_e(j)\le C_3 \sqrt{\frac{n}{m_2}}=
C_3
\prod_{i=1}^2\Bigl(1\wedge \sqrt{\frac{n}{m_i}}\Bigr).
\end{equation}

\smallskip
\textit{Case 3: $n< m_1\le m_2$.}
Using the fact that
$
\Bigl(1\wedge \frac{j}{\sqrt{m_1}}\Bigr)
\Bigl(1\wedge \frac{j}{\sqrt{m_2}}\Bigr)\le \frac{j^2}{\sqrt{m_1m_2}}
$
and \eqref{eq:mu-upper-per}, we obtain
\begin{equation}\label{eq:sum-case3}
\P(e_1,e_2\in\mathcal C_{\rm CP}(\rho)\mid e\in\mathcal C_{\rm CP}(\rho))
\le \frac{C_1}{\sqrt{m_1m_2}} \sum_{j\ge1} j^2\,\vartheta_e(j) \le C_4\,\frac{n}{\sqrt{m_1m_2}}
=
C_4
\prod_{i=1}^2\Bigl(1\wedge \sqrt{\frac{n}{m_i}}\Bigr).
\end{equation}

Combining \eqref{eq:sum-case1}, \eqref{eq:sum-case2}, and
\eqref{eq:sum-case3}, we conclude from \eqref{eq:sum-to-bound} that
\begin{equation}\label{eq:joint-upper-product-form}
\P\bigl(e_1,e_2\in \mathcal C_{\rm CP}(\rho)\mid e\in \mathcal C_{\rm CP}(\rho)\bigr)
\le
C_5
\prod_{i=1}^2\Bigl(1\wedge \sqrt{\frac{n}{m_i}}\Bigr).
\end{equation}

It remains to compare the right-hand side with the product of $\P(e_i\in \mathcal C_{\rm CP}(\rho)\mid e\in \mathcal C_{\rm CP}(\rho))$ for $i\in \{1,2\}$. Since $e_i$ is below $e$, we have
$
\{e_i\in \mathcal C_{\rm CP}(\rho)\}\subseteq \{e\in \mathcal C_{\rm CP}(\rho)\}
$
for $i\in \{1,2\}$,
and hence
\begin{equation}\label{eq:cond-marginal-ratio}
\P\bigl(e_i\in \mathcal C_{\rm CP}(\rho)\mid e\in \mathcal C_{\rm CP}(\rho)\bigr)
=
\frac{\P(e_i\in \mathcal C_{\rm CP}(\rho))}{\P(e\in \mathcal C_{\rm CP}(\rho))}.
\end{equation}
By Proposition~\ref{thm:ruin-order}, there
exist constants $c_1>0$ and $C_6<\infty$ such that for all $f\in E$,
$$
c_1\,|f|^{-1/2}\le \P(f\in \mathcal C_{\rm CP}(\rho))=r_{|f|}\le C_6\,|f|^{-1/2}.
$$
Since $|e_i|=n+m_i-1$, applying this to \eqref{eq:cond-marginal-ratio}, we obtain
\begin{align}
\P\bigl(e_i\in \mathcal C_{\rm CP}(\rho)\mid e\in \mathcal C_{\rm CP}(\rho)\bigr)
&\ge
c_2\,\sqrt{\frac{n}{n+m_i-1}}
\quad \text{for }i\in \{1,2\}.
\label{eq:marginal-lower}
\end{align}
Finally, for every $m\ge 1$ and $n\ge 1$, we note that
$1\wedge \sqrt{\frac{n}{m}}
\le
\sqrt{2}\,\sqrt{\frac{n}{n+m-1}}$. Combining this inequality with \eqref{eq:joint-upper-product-form} and
\eqref{eq:marginal-lower}, we get
\begin{align*}
\P\bigl(e_1,e_2\in \mathcal C_{\rm CP}(\rho)\mid e\in \mathcal C_{\rm CP}(\rho)\bigr)
&\le
C_7
\prod_{i=1}^2 \sqrt{\frac{n}{n+m_i-1}}\\
&\le
M\,
\P\bigl(e_1\in \mathcal C_{\rm CP}(\rho)\mid e\in \mathcal C_{\rm CP}(\rho)\bigr)
\P\bigl(e_2\in \mathcal C_{\rm CP}(\rho)\mid e\in \mathcal C_{\rm CP}(\rho)\bigr),
\end{align*}
for some constant $M<\infty$ independent of $e_1, e_2$. This proves
\eqref{eq:edge-qi-final}.
\end{proof}

We now combine the asymptotic behavior of ruin probabilities with the quasi-independent percolation criterion to prove the main theorem.

\begin{proof}[Proof of Theorem~\ref{thm:main}]

By Lemma~\ref{per2.lem},
\[
\P(\tau_\rho^+=\infty)
=
\P\bigl(|\mathcal C_{\rm CP}(\rho)|=\infty\bigr),
\]
where $\mathcal C_{\rm CP}(\rho)$ is the cluster of the root in the ruin
percolation. Consequently, the TSAW is transient if and only if the cluster
$\mathcal C_{\rm CP}(\rho)$ is infinite with positive probability, and recurrent
if and only if it is almost surely finite.

By Proposition~\ref{prop:per}, the ruin percolation is quasi‑independent, so the
criteria of Proposition~\ref{lyons1989} apply. Fix an edge $e=\{e^{-},e^{+}\}\in E$ with $|e|\ge 1$. By definition,
\[
\mathbf Q(\rho\leftrightarrow e^{+})
=
\P\bigl(\tau_{e^{+}}(e^{+})<\tau_\rho^+(e^{+})\bigr).
\]
Since $\X^{(e^{+})}$ has the same law as TSAW on $\{0,1,2,\cdots,|e|\}$, it follows that
$$\mathbf Q(\rho\leftrightarrow e^{+})=r_{|e|},$$
where we recall that $r_n$ is the ruin probability that the TSAW on $\{0, 1,2,\cdots, n\}$  hits $n$ before returning to $0$. 
By Proposition~\ref{thm:ruin-order}, there exists a constant $c_*\in (0,\infty)$ such that $r_n\sim c_*\,n^{-1/2}$ as $n\to\infty$. Hence,
\begin{equation}\label{eq:Q-asymptotic}
\mathbf Q(\rho\leftrightarrow e^{+})=\P\bigl(\tau_{e^{+}}<\tau_\rho^+\bigr)
\sim c_* |e|^{-1/2}.
\end{equation}

Recall the definition of the adapted conductances:
\[
c(e)
=
\begin{cases}
1, & |e|=1,\\[1ex]
\dfrac{\mathbf Q(\rho\leftrightarrow e^{+})}
      {\mathbf Q(e\ \text{is closed}\mid \rho\leftrightarrow e^{-})},
& |e|>1.
\end{cases}
\]
Since
\[
\mathbf Q(e\ \text{is closed}\mid \rho\leftrightarrow e^{-})
=
1-
\frac{\mathbf Q(\rho\leftrightarrow e^{+})}
     {\mathbf Q(\rho\leftrightarrow e^{-})},
\]
and by \eqref{eq:Q-asymptotic},
\[
\frac{\mathbf Q(\rho\leftrightarrow e^{+})}
     {\mathbf Q(\rho\leftrightarrow e^{-})}
=
1 - \frac{1}{2|e|}+o(|e|^{-1}),
\]
we obtain
\[
\mathbf Q(e\text{ is closed}\mid \rho\leftrightarrow e^-)
=
\frac{1}{2|e|}+o(|e|^{-1}).
\]
Combining this with \eqref{eq:Q-asymptotic}, we get
\begin{equation}\label{eq:c-asymptotic}
c(e)\sim 2c_*\,|e|^{1/2}.
\end{equation}

We distinguish the following two cases:

\medskip

\textbf{Case 1: when ${\rm br}_r(\Tcal)<\tfrac12$.}
Assume ${\rm br}_r(\Tcal)<\frac12$.
Choose $\gamma$ such that $
{\rm br}_r(\Tcal)<\gamma<\tfrac12.$ By the definition of the branching‑ruin number, 
\[
\inf_{\pi\in\Pi}\sum_{e\in\pi}|e|^{-\gamma}=0.
\]
Since $|e|^{-1/2}\le |e|^{-\gamma}$ for $\gamma<\tfrac12$, and using
\eqref{eq:Q-asymptotic}, we have
\[
\inf_{\pi\in\Pi}\sum_{e\in\pi}\mathbf Q(\rho\leftrightarrow e^{+})
=
0.
\]
By Proposition~\ref{lyons1989}(i), $\mathbf Q(\rho\leftrightarrow \infty)=0$ and the cluster $\mathcal{C}_{\rm CP}(\rho)$ is thus almost surely finite. Lemma~\ref{per2.lem} therefore
implies that the TSAW is a.s. recurrent.

\medskip

\textbf{Case 2: when ${\rm br}_r(\Tcal)>\tfrac12$.}
Assume ${\rm br}_r(\Tcal)>\frac12$, and fix $\gamma$ such that $
\tfrac12<\gamma<{\rm br}_r(\Tcal)$.
By the max-flow min-cut Theorem, there exists a non-zero flow $(\theta(e))_{e\in E}$ such that
$
\theta(e)\le |e|^{-\gamma}.
$
By \eqref{eq:c-asymptotic}, we also have $
c(e)\sim 2c_*\,|e|^{1/2}$. Hence 
$$\sup_{v\in V}\sum_{e\in \mathcal{P}_v}\frac{\theta(e)}{c(e)}\le C \sum_{n=1}^{\infty}\frac{1}{n^{\gamma+1/2}}<\infty.$$
Hence
\[
\sum_{e\in E}\frac{\theta(e)^2}{c(e)}\le \int_{\partial \Tcal} V_{\theta}(\xi)\rmd \mathfrak{m}_{\theta}(\xi) < \infty.
\]
where $V_{\theta}(\xi):=\sum_{e\in \xi} \frac{\theta(e)}{c(e)}$ for each $\xi\in \partial \Tcal$ and $\mathfrak{m}_\theta$ is the harmonic measure induced by flow $\theta$ on $\partial \Tcal$ (see e.g., Proposition 16.1 in \cite{LP2016}).
Therefore, by
Proposition~\ref{lyons1989}(ii), we get
$\mathbf Q(\rho\leftrightarrow \infty)>0$, and the cluster $\mathcal{C}_{\rm CP}(\rho)$ is thus infinite with positive probability.
Lemma~\ref{per2.lem} therefore
implies that the TSAW is a.s. transient.
\end{proof}

\section*{Acknowledgment} 
The author would like to thank Andrea Collevecchio for mentioning this problem. Tuan-Minh Nguyen was partially supported by the Australian Research Council under grant ARC DP230102209.

\bibliography{refs}
\bibliographystyle{plain}
\end{document}